\documentclass[a4paper,11pt]{amsart}

\usepackage{geometry}
\usepackage{amssymb,amsmath,amsthm,amscd}
\usepackage[ansinew]{inputenc} %per le accentate
\usepackage[english]{babel}

\usepackage{newtxtext}
\usepackage{newtxmath}
\usepackage[T1]{fontenc}
\usepackage{mathrsfs}
\usepackage{amscd}
\usepackage{currfile}
\usepackage{bm}
\usepackage{graphicx}
\usepackage{multirow}
\usepackage{hyperref}

\newtheorem{theorem}{Theorem}%[section]
\newtheorem*{theorem*}{Theorem}
\newtheorem{definition}[theorem]{Definition}
\newtheorem{lemma}[theorem]{Lemma}
\newtheorem{proposition}[theorem]{Proposition}
\newtheorem{corollary}[theorem]{Corollary}
\newtheorem{example}[theorem]{Example}

\theoremstyle{definition}
\newtheorem{remark}[theorem]{Remark}

\newcommand{\oo}{{\mathbb{O}}}
\newcommand{\hh}{{\mathbb{H}}}
\newcommand{\HH}{{\mathbb{H}}}

\newcommand{\cc}{{\mathbb{C}}}
\newcommand{\rr}{{\mathbb{R}}}

\newcommand{\nn}{{\mathbb{N}}}
\newcommand{\qq}{{\mathbb{Q}}}

\newcommand{\s}{{\mathbb{S}}}

\newcommand{\dB}{\overline\partial_{\B}}
\newcommand{\GB}{\Gamma_\B}

\newcommand{\B}{\mathcal{B}}

\newcommand{\sr}{\mathcal{SR}}
\newcommand{\srl}{\mathcal{SR}_{loc}}

\newcommand{\SC}{\mathcal{SC}}
\newcommand{\SL}{\mathcal{S}}

\newcommand{\I}{\mathcal{I}}

\newcommand{\dibar}{\overline\partial}
\newcommand{\dcf}{\dibar}
\newcommand{\dif}{\vartheta_\B}
\newcommand{\difbar}{\overline{\vartheta}_\B}
\newcommand\IM{\operatorname{Im}}
\newcommand\RE{\operatorname{Re}}
\newcommand\dds[1]{\partial_{#1}}

\newcommand\vs[1]{{#1}_s^\circ}
\newcommand\sd[1]{{#1}'_s}

\newcommand\Span{\operatorname{span}}

\newcommand{\ui}{\imath}

\newcommand{\OO}{\Omega}

\newcommand{\mbb}{\mathbb}

\newcommand{\R}{\mbb{R}}
\newcommand{\mc}{\mathcal}

\newcommand{\dd[3]}{\frac{\partial #2}{\partial #3}}
\newcommand\ddd[2]{\frac{\partial#1}{\partial#2}} 
 \newcommand{\C}{\mbb{C}}
 
 \newcommand{\Q}{\mc{Q}}
 \def\SS{\mathbb S}
 \newcommand\BB{\mathbb B}
 \newcommand{\bc}{\begin{center}}
 \newcommand{\ec}{\end{center}}

\newcommand{\Rn}{\mathbb R_n}
\newcommand{\cS}{c_m}
\newcommand{\Tdd}{\mathcal D_{\,\dibar\Delta}}
\newcommand{\Fdd}{\mathcal F_{\,\dibar\Delta}}

\newcommand{\AM}{\mathcal{AM}}

\begin{document}

\title{Cauchy-Riemann operators and local slice analysis over real alternative algebras}

\author{Alessandro Perotti}
\email{alessandro.perotti@unitn.it}
\address{Department of Mathematics, University of Trento, Via Sommarive 14, Trento Italy}
%\thanks{The author is a member the INdAM Research group GNSAGA and was supported by the grants ``Progetto di Ricerca INdAM, Teoria delle funzioni ipercomplesse e applicazioni'', and PRIN ``Real and Complex Manifolds: Topology, Geometry and holomorphic dynamics''.}

\begin{abstract}
We prove some formulas relating Cauchy-Riemann operators defined on hypercomplex subspaces of an alternative *-algebra to a differential operator associated with the concept of slice-regularity and to the spherical Dirac operator. These results in particular allow to introduce a definition of locally slice-regular function and open the path for local slice analysis. Since Cauchy-Riemann operators factor the corresponding Laplacian operators, the proven formulas let us also obtain several results about the harmonicity and polyharmonicity properties of slice-regular functions. 
\end{abstract}


\maketitle
% \begin{center}
% {\large\texttt{(\currfilename)}}
% \end{center}

\section{Introduction}
The first aim of this work is to develop a set of relations between two higher dimensional function theories (or better two families of function theories). These theories were born as different generalizations of that fundamental part of modern mathematics represented by the theory of holomorphic functions of one complex variable.  One function theory, which is well developed and dates back to the 1930’s, is based on a system of linear first-order constant coefficients differential equations, known as the Cauchy–Fueter system in the case of quaternions, and as the Dirac (or Cauchy-Riemann) system in the case of Clifford algebras. The functions in the kernel of these systems are called, respectively, Fueter-regular and monogenic functions. We refer the reader to the monographs \cite{BDS,GHS} for extended accounts of this theory. 

Over the last fifteen years, another higher dimensional function theory, better adapted to algebraic requirements, in particular the inclusion of polynomials, was developed. 
%The search for an approach better adapted to algebraic requirements, e.g., the inclusion of the classical theory of polynomials, led to the development, over the last fifteen years, of another higher dimensional theory. 
Born in the quaternionic setting \cite{GeSt2006CR,GeSt2007Adv}, the theory of \emph{slice-regular functions}, also called \emph{slice analysis}, has then been  extended to the octonions,  to Clifford algebras, and more generally to real alternative *-algebras (see, e.g., \cite{CoSaSt2009Israel,GeStRocky,AIM2011}). See also \cite{GeStoSt2013,Struppa2015} for reviews of this function theory and extended references. 

The relations between the two theories we are presenting here are based on some fundamental formulas linking  Cauchy-Riemann operators, %defined on hypercomplex subspaces of the algebra, 
which are at the core of the first function theory, and two other differential operators related to slice-regularity and to the spherical derivative of a slice function (Theorem \ref{teo:difference} and Proposition \ref{pro:gamma_slice}). 
A natural setting where Cauchy-Riemann operators can be defined are \emph{hypercomplex subspaces} of the algebra. 
Relevant examples of hypercomplex subspaces of an alternative *-algebra are: the quaternionic space $\hh$, the space $\oo$ of octonions, %(these are the unique examples where the subspace coincides with the whole algebra), 
the reduced quaternions and more generally the paravector subspace of the real Clifford algebra $\R_n$ of signature $(0,n)$.

One of the peculiar aspects of slice analysis %as developed till now 
is the non-local character of the existing definitions of slice-regular function. Both in the original approach \cite{GeSt2007Adv}, with functions defined on \emph{slice domains}, %by definition 
and then on open sets intersecting the real axis, and in the \emph{stem function} approach \cite{AIM2011}, where functions are defined on domains that are axially symmetric around the real axis, a local definition of slice-regular function is not feasible. The above-mentioned relations with Cauchy-Riemann operators permit to refine the original definition of slice-regularity and find a formulation that has a natural local version. 
This is a second major aim of this work. We are able to give a definition of \emph{locally slice-regular function} (Definition\ \ref{def:ssr}) on any open subset of an hypercomplex subspace of the algebra that is compatible with the existing definitions. In the setting of slice analysis on hypercomplex subspaces, local slice-regularity can be seen as the most faithful generalization of the classic concept of holomorphy.

As an application of this local approach, we present some results of quaternionic \emph{local slice analysis}. In particular we give local versions of some differential topological properties of slice-regular functions. %Here we restrict to the case of the quaternionic algebra, where the results are 
Of course this type of results is not restricted to the quaternionic setting. Thanks to a local extendibility theorem (Theorem \ref{thm:localext}), every local property satisfied by slice-regular functions, originally proved on axially symmetric domains, remains valid for locally slice-regular functions defined on any open subset of an hypercomplex subspace of the algebra. 

We also investigate harmonicity and polyharmonicity properties of slice-regular functions (Theorem \ref{thm:n_odd_even}), extending to any alternative algebra what already known for quaternions and Clifford algebras. 
We apply the results proved here to obtain a decomposition of any slice-regular function in terms of components in the kernel of the third order operator $\dibar\Delta$, where $\dibar$ is the Cauchy-Riemann operator of the hypercomplex subspace $M$ and $\Delta$ is the corresponding Laplacian (Theorem \ref{thm:dDeltaDec}). If the dimension $m+1$ of $M$ is even, this decomposition  defines a real-linear map that associates to any slice-regular function a $(m-1)/2$-tuple of axially monogenic functions (\emph{$\dibar\Delta$-Fueter mapping}). This map represents a new generalization of the quaternionic Fueter's Theorem (corresponding to the case $m=3$), valid on any hypercomplex subspace and gives another strong link between the function theories. We recall that when $M$ is the paravector subspace of $\R_m$, then the Sce's generalization \cite{Sce} of Fueter's Theorem  provides a mapping, namely the power $\Delta^{(m-1)/2}$ of the Laplacian, that associates to any slice-regular function a single axially monogenic function.  We prove that, differently from higher power of the $\Delta$ operator, the $\dibar\Delta$-Fueter mapping has a small kernel for every $m\ge3$ (Proposition\ \ref{pro:Kernel}). 

We describe in more detail the structure of the paper. 
In Section\ \ref{sec:pre}, we recall some basic definitions about real alternative *-algebras, slice functions and slice-regularity.  
In Section\ \ref{sec:CR_operators},  we give the definition of hypercomplex subspace $M$ in an alternative *-algebra $A$, and introduce the three fundamental differential operators on $M$: the Cauchy-Riemann operator $\dB$, the global operator $\difbar$ associated to slice-regularity and the spherical Dirac operator $\Gamma_\B$. Then we establish the main relation linking the three operators and relate $\Gamma_\B$ and $\dB$ to the slice and spherical derivatives of a slice function.   
Section\ \ref{sec:strong_local} is devoted to the definitions of \emph{strongly slice-regular function} and \emph{locally slice-regular function} and to proving two fundamental extendibility theorems: the global one and the corresponding local version. Subsection\ \ref{sec:quatlocal} contains some results of local slice analysis on the quaternionic algebra: the Quasi-open Mapping Theorem, the Mean value and Poisson formulas. 
In Section\ \ref{sec:Laplacian} we study the action of the Laplacian on slice-regular functions and show that every (locally) slice-regular function is $(m+1)/2$-polyharmonic, with $m+1$ the dimension of $M$. We then prove the $\dibar\Delta$-decomposition of a slice-regular function combining the classical Almansi's Theorem for polyharmonic functions with a polyharmonic zonal decomposition. By composition with the Laplacian operator we obtain the $\dibar\Delta$-Fueter mapping. 
In Subsection\ \ref{sec:different} we prove a second polyharmonic decomposition of slice-regular functions. We conclude the section with a list of examples of the previous decompositions in different settings: in a Clifford algebra, in the octonion space or in the space of reduced quaternions. The Appendix contains the proofs of Theorem\ \ref{teo:difference} and Proposition\ \ref{pro:powers}.

\section{Preliminaries}\label{sec:pre}

Let $A$ be a real algebra with unity $1\ne0$. The real multiples of $1$ in $A$ are identified with the field $\R$ of real numbers. Assume that $A$ is alternative, i.e., the \emph{associator} $(x,y,z) := (xy)z - x(yz)$ of three elements of $A$ is an alternating function of its arguments.
Alternativity yields the so-called Moufang identities:
\begin{eqnarray}
(xax)y &=& x(a(xy))\label{moufang1}\\
y(xax) &=& ((yx)a)x\label{moufang2}\\
(xy)(ax) &=& x(ya) x.\label{moufang3}
\end{eqnarray}
A theorem of E.\ Artin asserts that the subalgebra generated by any two elements of $A$ is associative. Assume that $A$ is a *-algebra, i.e., it is equipped with a real linear anti-involution $A\to A$, $x\mapsto x^c$, such that $(xy)^c=y^cx^c$ for all $x,y\in A$ and  $x^c=x$ for $x$ real. Let $t(x):=x+x^c\in A$ be the \emph{trace} of $x$ and $n(x):=xx^c\in A$  the \emph{(squared) norm} of $x$. 
Let
\[
\s_A:=\{J\in A : t(x)=0,\ n(x)=1\}
\]
be the `sphere' of the imaginary units of $A$ compatible with the *-algebra structure of $A$. 
Assuming $\s_A\ne\emptyset$, one can consider the \emph{quadratic cone} of $A$ (see \cite[Def.~3]{AIM2011}), defined as the subset of $A$
\[
Q_A:=\bigcup_{J\in \s_A}\C_J,
\]
where $\C_J=\Span(1,J)$ is the complex `slice' of $A$ generated by $1$ and $J$. It holds $\C_J\cap\C_K=\R$ for each $J,K\in\s_A$ with $J\ne\pm K$. The quadratic cone is a real cone invariant w.r.t.\ translations along the real axis. 
%Here we write `w.r.t.' to abbreviate `with respect to'.
Observe that $t$ and $n$ are real-valued on $Q_A$. %and that $Q_A=A$ if and only if $A$ is isomorphic as a real $^*$-algebra to one of the division algebras $\C,\HH,\mathbb O$ with the standard conjugations (see \cite[Prop.~1]{AIM2011}).
Moreover, it holds
\[
Q_A=\R\cup\{x\in A\setminus\R: t(x)\in\R,n(x)\in\R,4n(x)>t(x)^2\}.
\]
Each element $x$ of $Q_A$ can be written as $x=\RE(x)+\IM(x)$, with $\RE(x)=\frac{x+x^c}2$, $\IM(x)=\frac{x-x^c}2=\beta J$, where $\beta=\sqrt{n(\IM(x))}\geq0$ and $J\in\s_A$, with unique choice of $\beta\geq0$ and $J\in\s_A$ if $x\not\in\R$.

Observe that the quadratic cone $Q_A$ is properly contained in $A$ unless $A$ is isomorphic as a real *-algebra to one of the division algebras $\C,\HH,\mathbb O$ with the standard conjugations (see \cite[Prop.~1]{AIM2011}).
We refer to \cite{GhPe_Trends}, \cite[\S2]{AIM2011} and \cite[\S1]{AlgebraSliceFunctions} for more details and examples about real alternative *-algebras and their quadratic cones.

%%%

\subsection{Slice functions and slice-regular functions}

The functions on $A$ which are compatible with the slice character of the quadratic cone are called \emph{slice functions}. More precisely, let $A\otimes_{\R}\C$ be the complexified algebra, whose elements $w$ are of the form $w=a+\ui b$ with $a,b\in A$ and $\ui^2=-1$. In $A\otimes_{\R}\C$ we consider the complex conjugation mapping $w=a+\ui b$ to $\overline w=a-\ui b$  for all $a,b\in A$. Let $D$ be a subset of $\C$ that is invariant w.r.t.\ complex conjugation. 
If a function $F: D \to A\otimes_{\R}\C$ satisfies  $F(\overline z)=\overline{F(z)}$ for every $z\in D$, then $F$  is called a \emph{stem function} on $D$. For every $J\in\s_A$, we define the *-algebra isomorphism $\phi_J:\C\to\C_J$ by setting
\[
\phi_J(\alpha+i\beta):=\alpha+J\beta\text{\quad for all $\alpha,\beta\in\R$}.
\]
Let $\OO_D$ be the \emph{axially symmetric} (or \emph{circular}) subset of the quadratic cone defined by 
\[
\OO_D=\bigcup_{J\in\s_A}\phi_J(D)=\{\alpha+J\beta\in
A : \alpha,\beta\in\R, \alpha+i\beta\in D,J\in\s_A\}.
\]
An axially symmetric connected set $\OO=\OO_D$ is called a \emph{slice domain} if $\OO\cap \R\ne\emptyset$, a \emph{product domain} if $\OO\cap \R=\emptyset$. Any axially symmetric open set is union of a family of domains of these two types. 

The stem function $F=F_1+\ui F_2:D \to A\otimes_{\R}\C$  induces the \emph{(left) slice function} $f=\I(F):\OO_D \to A$ in the following way: if $x=\alpha+J\beta =\phi_J(z)\in \OO_D\cap \C_J$, then  
\[ f(x)=F_1(z)+JF_2(z),\text{\quad where $z=\alpha+i\beta$}. 
\]

Suppose that $D$ is open.  The slice function $f=\I(F):\OO_D \to A$ is called \emph{(left) slice-regular} if $F$ is holomorphic w.r.t.\ the complex structure on $A\otimes_{\R}\C$ defined by left multiplication by $\ui$. 
We will denote by $\SL^1(\OO)$ the real vector space of slice functions induced by stem functions of class $\mathcal C^1$ on $\OO$ and by $\sr(\OO)$ the vector subspace of slice-regular functions on $\OO$. 
For example, polynomial functions $f(x)=\sum_{j=0}^d x^ja_j$ and convergent power series with right coefficients in $A$ are slice-regular.
If $A=\hh$ and $\OO_D$ is a slice domain, this definition of slice regularity is equivalent to the original one proposed by Gentili and Struppa in \cite{GeSt2007Adv}. 

To any slice function $f=\I(F):\OO_D \to A$, one can associate the function $\vs f:\OO_D \to A$, called \emph{spherical value} of $f$, and the function $f'_s:\OO_D \setminus \R \to A$, called  \emph{spherical derivative} of $f$, defined as
\[
\vs f(x):=\frac{1}{2}(f(x)+f(x^c))
\quad \text{and} \quad
f'_s(x):=\frac{1}{2}\IM(x)^{-1}(f(x)-f(x^c)).
\] 
If $x=\alpha+\beta J\in\OO_D$ and $z=\alpha+i\beta\in D$, then $\vs f(x)=F_1(z)$ and $f'_s(x)=\beta^{-1} F_2(z)$. Therefore $\vs f$ and $f'_s$ are slice functions, constant on every set $\s_x:=\alpha+\beta\,\s_A$. They are slice regular only if $f$ is locally constant.
Moreover, the formula
\[
f(x)=\vs f(x)+\IM(x)f'_s(x)
\]
holds for all $x\in\OO_D\setminus \R$. In the case $A=\HH$, a remarkable property of the spherical derivative of a slice-regular function is its harmonicity in the four real variables. In the higher dimensional case, a polyharmonicity property holds (see \S\ref{sec:CR_operators}).

In general, the pointwise product $x \mapsto f(x)g(x)$ of slice functions $f=\I(F)$ and $g=\I(G)$ is not a slice function. On the other hand, it is easy to verify that the pointwise product $FG$ of stem functions $F$ and $G$ is again a stem function. %This fact suggested the following definition.
The \emph{slice product} of two slice functions $f=\I(F)$, $g=\I(G)$ on $\OO=\OO_D$ is defined by means of the pointwise product of the stem functions:
\[f\cdot g=\I(FG).\]
% The antiinvolution $f\mapsto f^c$ and the \emph{normal function} $N(f)$ are defined by
% \[f^c=\I(F^c),\qquad N(f)=f\cdot f^c\]
% respectively. 
% If $f,g\in\mc{SR}(\OO)$, then $f\cdot g$,  $f^c$ and  $N(f)\in\mc{SR}(\OO)$.
In the case of slice-regular functions, this product is also called \emph{star product} of $f$ and $g$, denoted by $f\ast g$.

The function $f=\I(F)$ is called \emph{slice-preserving} if the $A$-components $F_1$ and $F_2$ of the stem function $F$ are real-valued.
%This is equivalent to the condition $f(\overline x)=\overline{f(x)}$ for every $x\in\OO$.
If $f$ is slice-preserving, then $f\cdot g$ coincides with the pointwise product of $f$ and $g$. %In this case, we will denote it simply by $fg$. %but this is not true in general.
%The antiinvolution $f\mapsto f^c=\I(F^c)$ defines the \emph{normal function} $N(f)=f\cdot f^c$ of $f$. %, which is always slice-preserving. 
If $f,g$ are slice-regular on $\OO$, then also their slice product $f\cdot g$ is slice-regular on $\OO$. 

The \emph{slice derivatives} $\dd{f}{x},\dd{f\;}{x^c}$ of a slice functions $f=\I(F)$ are defined by means of the  Cauchy-Riemann operators applied to the inducing stem function $F$:
\[\dd{f}{x}=\I\left(\dd{F}{z}\right),\quad \dd{f\;}{x^c}=\I\left(\dd{F}{\overline z}\right).\]
It follows that $f$ is slice-regular if and only if $\dd{f\;}{x^c}=0$ and if $f$ is slice-regular on $\OO$ then also $\dd{f}{x}$ is slice-regular on $\OO$. Moreover, the slice derivatives satisfy the Leibniz product rule w.r.t.\ the slice product. If $f=\dd g x$, we will say that $g$ is a \emph{slice primitive} of $f$.

A slice-regular function $f$ is called \emph{slice-constant} if $\dd{f}x=0$ on $\OO$. 
We will denote by $\SC(\OO)$ the real vector space of slice-constant functions on $\OO$. If $\OO$ is a slice domain, then every $f\in\SC(\OO)$ is a  constant function. If $\OO$ is a product domain, then other possibilities arise (see e.g.\ \cite[Remark 12]{AIM2011}). Observe that $f\in\SC(\OO)$ if and only if $f$ is locally constant on slices $\OO\cap\cc_J$ or, equivalently, $f$ is induced by a locally constant stem function. 

We refer the reader to \cite[\S3,4]{AIM2011} for more properties of slice functions and slice regularity.

\section{Cauchy-Riemann operators on hypercomplex subspaces}\label{sec:CR_operators}

We recall a concept introduced in \cite{VolumeCauchy}.
A non--empty subset $S$ of $A$ is a \emph{genuine imaginary sphere of $A$} (a \emph{gis} of $A$), if there exists a real vector subspace $M$ of $A$ such that $\R \subset M \subset \Q_A$ and $S=M \cap \s_A$. If such a $M$ exists, then it is unique, since $M=\bigcup_{J \in S}\C_J$. We call $M$ the \textit{vector subspace inducing $S$}. If $\dim(M)>2$, such a subspace $M$ of $A$ will be called a \emph{hypercomplex subspace of $A$}. 

For example, if $A$ is the skew-field $\hh$ of quaternions or the space $\oo$ of octonions, then $A$ itself is an hypercomplex subspace, with $S=\s_A$. These examples are the unique alternative algebras, except for $\cc$, which have quadratic cone $\Q_A$ equal to the whole algebra $A$ \cite[Prop.1]{AIM2011}. The algebra $\hh$ contains also the hypercomplex subspace of \emph{reduced quaternions} $\hh_r=\{x=x_0+ix_1+jx_2\in\hh\;|\; x_0,x_1,x_2\in\R\}$. More generally, the space $M=\rr^{n+1}$ of \emph{paravectors} in the real Clifford algebra $\R_n$ of signature $(0,n)$ is an hypercomplex subspace of $\Rn$.  

It is easy to verify that every hypercomplex subspace $M$ of $A$ is, in particular, a \emph{strong regular quadratic cone} in $A$, as defined in \cite{RenWangXu}.

\begin{lemma}\cite[Lemma 1.4]{VolumeCauchy} \label{lem:gis}
Let $S$ be a gis of $A$ and let $M$ be the vector subspace inducing $S$. Then there exists a norm $\|\ \|$ on $A$ such that $\|x\|^2=n(x)\ \forall x \in M$.
\end{lemma}

%\marginpar{$m->m+1$? No - cf. Sec/Qian}{}
Let $\B=(v_0,v_1,\ldots,v_{m})$ be a real vector basis of $M$ with $v_0=1$, orthonormal w.r.t.\ the scalar product $\langle\ ,\,\rangle$ on $A$ associated to the norm $\|\ \|$. Complete $\B$ to a real vector basis $\B_A=(v_0,v_1,\ldots,v_{d-1})$ of $A$, orthonormal w.r.t.\ the scalar product $\langle\ ,\,\rangle$ on $A$.  
For every $x,y\in M$, it holds $2\langle x,y\rangle=\|x+y\|^2-\|x\|^2-\|y\|^2=n(x+y)-n(x)-n(y)= t(xy^c)$. Therefore $\langle x,y\rangle=\frac12 t(xy^c)$ on $M$. 
In particular, 
\begin{align}
&t(v_i)=2\langle v_i,1\rangle=0\quad\text{for every $i=1,\ldots,m$,} \label{eq:t1}\\
&t(v_iv_j)=-t(v_iv_j^c)=-2\langle v_i,v_j\rangle=0\quad\text{for every $i\ne j$ in the set $\{1,\ldots,m\}$.}\label{eq:t2}
\end{align}
Property \eqref{eq:t1} implies that $v_i^c=-v_i$ for every $i=1,\ldots,m$ and property \eqref{eq:t2} shows that for every $i\ne j$ in the set $\{1,\ldots,m\}$, the elements $v_i$ and $v_j$ anticommute, since $v_iv_j+v_jv_i=t(v_iv_j)=0$. Moreover, $v_i^2=-v_iv_i^c=-n(v_i)=-\|v_i\|^2=-1$ for every $i=1,\ldots,m$. 
By Proposition 1 of \cite{AIM2011}, we know that $\s_A=\{J \in A \, | \, t(J)=0, \ n(J)=1\}$. Therefore $\{v_1,\ldots,v_m\}\subseteq S$ and $M$ is the orthogonal direct sum of $\R=\Span(v_0)$ and $M \cap \ker(t)=\Span(v_1,\ldots,v_m)$.
Moreover, every triple $\{v_i,v_j,v_iv_j\}$, with $i\ne j$, $1\le i,j\le m$, is an \emph{Hamiltonian triple} in $A$, such that $\Span(1,v_i,v_j,v_iv_j)\simeq\hh$.

The Cauchy-Schwarz inequality for the scalar product $\langle\ ,\ \rangle$ implies that $xy^c$ belongs to the quadratic cone $\Q_A$ for every $x,y\in M$, since $t(xy^c)\in\R$, $n(xy^c)=(xy^c)(yx^c)=n(x)n(y)\in\R$ (see e.g.\ \cite[Thm.\ 1.7]{AlgebraSliceFunctions}), and it holds
\[
(t(xy^c))^2=(2\langle x,y\rangle)^2\le4\| x\|^2\| y\|^2=4n(x)n(y)=4n(xy^c),
\] 
with equality if and only if $xy^c$ is real.

Let $L:\R^d \to A$ be the real vector isomorphism sending $x=(x_0,x_1,\ldots,x_{d-1})$ into $L(x)=\sum_{\ell=0}^{d-1}x_{\ell}v_{\ell}$. Identify $\R^d$ with $A$ via~$L$ and $M$ with $\R^{m+1}=\R^{m+1} \times \{0\} \subset \R^d$. The product of $A$ becomes a product on $\R^d$. Given $x,y \in \R^d$, $xy$ is defined as $L^{-1}(L(x)L(y))$. Since $\B_A$ is orthonormal, $\|x\|$ coincides with the Euclidean norm $(\sum_{\ell=0}^{d-1}x_{\ell}^2)^{1/2}$ of $x$ in $\R^d$. Moreover,
\[
S=\{L(x)\in M\,|\,x\in\R^{m+1},\,x_0=0,\, \textstyle\sum_{i=1}^mx_i^2=1\}.
\]

\begin{definition}\label{def:db}
Let $\OO$ be an open subset of the vector subspace $M$ inducing $S$. The \emph{Cauchy-Riemann operator induced by %the basis 
$\B$%of $M$
} is the partial differential operator with constant coefficients $\dB:C^1(\OO,A)\to C^0(\OO,A)$ defined by
\[
\dB:=\tfrac12\left(\partial_0+v_1\partial_1+\cdots +v_m\partial_m\right),
\]
where the operators $\partial_i:C^1(\OO,A)\to C^0(\OO,A)$ are defined by 
$\partial_i f=L\circ\dd {(L^{-1}\circ f\circ L)}{x_i}\circ L^{-1}$ for $i=0,1,\ldots,m$.
\end{definition}

\begin{remark}
Note that sometimes in the literature the factor $1/2$ is absent.
\end{remark}

\begin{remark}
From the definition, it is immediate to check that $\partial_i\partial_jf=\partial_j\partial_if$ for every $i,j$ and every $f$ of class $C^2$.
Moreover, the operators $\partial_i$ satisfy a Leibniz rule with respect to the pointwise product of $A$-valued functions. This can be easily seen introducing the structure constants of $A$ w.r.t.\ $\B$, i.e., the real numbers $\{c_{\alpha,\beta}^\gamma\}_{\alpha,\beta,\gamma}$ such that
\[v_\alpha v_\beta=\sum_{\gamma=0}^{d-1}c_{\alpha,\beta}^\gamma v_\gamma\text{\quad for every $\alpha,\beta\in\{0,\ldots,d\}$}.
\]
Let $\OO=\OO'\cap M$, with $\OO'$ open subset of $A$. Let $f(x)=\sum_{\alpha=0}^{d-1}f_\alpha(x)v_\alpha$ and $g(x)=\sum_{\beta=0}^{d-1}g_\beta(x)v_\beta$ be $A$-valued functions of class $C^1$ on $\OO'$, with $f_\alpha, g_\beta\in C^1(\OO',\R)$ for every $\alpha,\beta$. Then
$f(x)g(x)=\sum_{\gamma=0}^{d-1}\sum_{\alpha,\beta=0}^{d-1}f_\alpha(x)g_\beta(x)c^\gamma_{\alpha,\beta}v_\gamma$.
Therefore 
\begin{align*}
\partial_i(fg)&=\sum_{\gamma=0}^{d-1}\sum_{\alpha,\beta=0}^{d-1}\dd{((f_\alpha\circ L)(g_\beta\circ L))}{x_i} c^\gamma_{\alpha,\beta}v_\gamma\\
&=\sum_{\gamma=0}^{d-1}\sum_{\alpha,\beta=0}^{d-1}\left(\dd{(f_\alpha\circ L)}{x_i}(g_\beta\circ L)+(f_\alpha\circ L)\dd{(g_\beta\circ L)}{x_i}\right)c^\gamma_{\alpha,\beta}v_\gamma\\
%&=\sum_{h=0}^{d-1}\sum_{l,m=0}^{d-1}\dd{((f_l\circ L))}{x_i}(g_m\circ L)c^h_{l,m}v_h+\sum_{h=0}^{d-1}\sum_{l,m=0}^{d-1}(f_l\circ L)\dd{((g_m\circ L))}{x_i}c^h_{l,m}v_h\\
&=(\partial_if)g+f(\partial_ig).
\end{align*}
In particular, it holds $\partial_i(fa)=(\partial_if)a$ and $\partial_i(af)=a(\partial_if)$ for every $a\in A$.
In general, when $A$ is non-commutative or non-associative, $\dB$ does not  satisfy a Leibniz rule. However, if $A$ is associative, then $\dB(fa)=(\dB f)a$ for every function $f$ and $a\in A$.
\end{remark}
%(x_0,\ldots,x_{d-1})

%Let $\cS:=\dB x$, where $x=x_0+\sum_{i=1}^{d-1}x_iv_i$ is the identity function on $A$. 
Let $\cS:=\dB(L(x_0,\ldots,x_{d-1}))=\dB(x_0+\sum_{i=1}^{d-1}x_iv_i)$. 
Since $v_i^2=-1$ for every $i=1,\ldots,m$, it holds $\cS=(1-m)/2$. Let
\[
\partial_\B:=\tfrac12\left(\partial_0-v_1\partial_1-\cdots -v_m\partial_m\right)
\]
be the \emph{conjugated Cauchy-Riemann operator induced by $\B$} and let $\Delta_\B$ the \emph{Laplacian operator on $M$ induced by $\B$}, acting on functions $f$ of class $C^2(\OO,A)$ as
\[\Delta_\B f:=\sum_{i=0}^m\partial_i(\partial_if).\]
%(independent of $\B$?)

\begin{proposition}\label{pro:powers}
Let $\OO$ be an open subset of the vector subspace $M$ inducing $S$.
\begin{itemize}\setlength\itemsep{0.5em}
	\item[(a)]
For every $n\in\nn$ and every $x\in M$, it holds:
\[%\label{eq:db}
\dB(x^n)=\cS\sd{(x^n)}=\tfrac{1-m}2 \sd{(x^n)}=\tfrac{1-m}2\sum_{k=0}^{n-1}x^{n-k-1}(x^c)^k.
\]
	\item[(b)]
	For every function $f$ of class $C^2(\OO,A)$, it holds:
\[
\partial_\B(\dB f)=\dB(\partial_\B f)=\tfrac14\Delta_\B f.
\]
%where $\Delta_\B f=\sum_{i=0}^m\partial_i(\partial_if)$ is the \emph{Laplacian operator on $M$ induced by $\B$}.
\end{itemize}	
\end{proposition}
We postpone the proof of the proposition to the Appendix to improve the reading of the paper.

Let $\OO$ be an open subset of the vector subspace $M$ inducing $S$. 
We recall from \cite{Gh_Pe_GlobDiff} the definition of the global differential operators $\dif,\difbar: C^1(\OO\setminus \R,A) \to C^0(\OO\setminus \R, A)$ associated with the slice derivatives. Since we are considering only functions defined on $M$, in the formula defining $\dif$ and $\difbar$ we can take %only 
the derivatives w.r.t.\ the first $m+1$ coordinates: 
%\marginpar{$\vartheta_\B$?}
\[
\dif=\frac12\bigg(\partial_0-\frac{\IM(x)}{n(\IM(x))} \sum_{i=1}^m x_i\partial_i\bigg),\quad
\difbar=\frac12\bigg(\partial_0+\frac{\IM(x)}{n(\IM(x))} \sum_{i=1}^m x_i\partial_i\bigg),
\]
where the operators $\partial_i:C^1(\OO,A)\to C^0(\OO,A)$ were given in Definition \ref{def:db}. For every slice function $f$ on $\OO$, it holds $\dif f=\dd{f}{x}$ and $\difbar f =\dd f{x^c}$  on $\OO\setminus\R$ (see \cite[Theorem 2.2]{Gh_Pe_GlobDiff}).

Observe that on a slice $\C_I\subseteq M$, the operator $\difbar$ coincides with the standard Cauchy-Riemann operator of $\C_I$.
In particular, when $\dim(M)=2$, then $M=\cc_I$ and the operators $\difbar$, $\dB$ are both equal to the Cauchy-Riemann operator of $M$.

For any $i,j$ with $1\le i,j\le m$, let $L_{ij}=x_i\dds j-x_j\dds i$ and let
\[\Gamma_\B=-\tfrac12\sum_{i,j=1}^m v_i(v_jL_{ij})
\] 
be the \emph{spherical Dirac operator} associated to the basis $\B$. %of the subspace $\Span(v_1,\ldots,v_m)$. 
It acts on $C^1$ functions $f$ as $\GB f=-\tfrac12\sum_{i,j=1}^m v_i(v_jL_{ij}f)$.
In the case of octonions, this operator has been introduced also in the recent paper \cite{SliceFueterRegular}. 
The operators $L_{ij}$ are tangential differential operators for the spheres $\s_x\cap M=\alpha+S\beta$, with $x=\alpha+I\beta\in\Q_A$. The next proposition extends Propositions 3.2 and 6.1 of \cite{Harmonicity} to the setting of $C^1$ (not necessarily slice) functions on open domains in $M$. To take care also of the non-associative case, we firstly need a  definition.

\begin{definition}
A function $f:\OO\to A$ of class $\mathcal{C}^1$ is said to be \emph{$M$-admissible} if $f(x)$ and $\dB f(x)$ belong to $M$ for all $x\in\OO$. 
\end{definition}

Observe that when $A=\hh$ or $A=\oo$, we can take $M=\Q_A=A$ (see \cite[Prop.1]{AIM2011}). In this case every $A$-valued function is $M$-admissible.

\begin{proposition}\label{pro:gamma}
Let $\OO$ be an open subset of the vector %hypercomplex 
subspace $M$ inducing $S$. Let $\dim(M)=m+1$. Then it holds:
\begin{itemize}\setlength\itemsep{0.5em}
\item[(a)]
$\Gamma_\B (x)=(m-1)\IM(x)$.
\item[(b)]
Let $f:\OO\to A$ be a $\mathcal{C}^1$ function. If $A$ is not associative, assume that $f$ is $M$-admissible. Then $\Gamma_B$ and multiplication by $x$ satisfy the following commutation relation
\[\Gamma_\B (xf(x))=(m-1)\IM(x)f(x)+ x^c\,\Gamma_\B f(x)\]
for all $x\in\OO$.
\end{itemize}
%\qed
\end{proposition}
We postpone also the proof of Proposition \ref{pro:gamma} to the Appendix.

We now establish the main relation linking the three operators $\dB$, $\difbar$ and $\Gamma_\B$ on an hypercomplex subspace $M$ of $A$. Also the proof of Theorem \ref{teo:difference} will be given in the Appendix.

\begin{theorem}\label{teo:difference}
Let $\OO$ be an open subset of the vector subspace $M$ inducing $S$. Let $f:\OO\to A$ be a $\mathcal{C}^1$ function. If $A$ is not associative, assume that $f$ is $M$-admissible. Then the following formula holds on $\OO\setminus\R$:
\[
\dB f-\difbar f=-(2\IM(x))^{-1} \Gamma_\B f.
\]
%\qed
\end{theorem}

The vector subspace $M$ inducing the gis $S$ is contained in the quadratic cone $\Q_{A}$. Therefore we can consider the restriction of a slice function on an axially symmetric domain $\OO_D$ to the subset $\OO=\OO_D\cap M$ of $M$. We call such a set $\OO$ an \emph{axially symmetric} open subset of $M$. If $\OO$ is also connected, we call it an \emph{axially symmetric domain} of $M$. 

Thanks to the representation formula (see e.g.~\cite[Prop.6]{AIM2011}), the restriction of a slice function to $\OO_D$ uniquely determines the function. We will therefore use the same symbol to denote the restriction. 

Proposition \ref{pro:powers} shows that the spherical derivative of a slice-regular polynomial $\sum_{n=0}^dx^na_n$ with coefficients in $A$, at least in the associative case, can be obtained by application of a differential operator, namely,  the Cauchy-Riemann operator. We now extend this result to all slice-regular functions.

\begin{proposition}\label{pro:gamma_slice}
Let $\OO_D$ be an axially symmetric open subset of $\Q_A$ and let $\OO:=\OO_D\cap M$, an axially symmetric open subset of $M$. Assume that $m=\dim(M)-1$ is greater than 1. Let $f:\OO\to A$ be a slice function of class $\mathcal{C}^1(\OO)$. If $A$ is not associative, assume that $f$ is $M$-admissible. Let $\cS=(1-m)/2$. Then the following properties hold on $\OO\setminus\R$:
\begin{itemize}\setlength\itemsep{0.5em}
	\item[(a)]
$\Gamma_\B f=(m-1)\IM(x)\sd f$.
\item[(b)]
$\dB f=\dd f{x^c}+\cS f'_s$,
\item[(c)] 
$\partial_\B f=\dd fx-\cS f'_s$\quad and\quad $\dd {(f'_s)}x=\partial_\B f'_s$.
\item[(d)]
$f$ is slice-regular if and only if $\dB f=\cS f'_s$.   
\item[(e)]
$f\in\ker(\dB)$ if and only if $\dd f{x^c}=-\cS\sd f$.
\item[(f)]
If $f$ is slice-regular, then 
%\marginpar{(10)1 suff.\ $f$ slice-harmonic}
\begin{equation}\label{eq:Delta}
\Delta_\B f=4\cS\,\dd{(\sd f)}x\quad\text{and}\quad \IM(x)\Delta_\B f=2\cS\left(\dd f x-\sd f\right).
\end{equation}
\item[(g)]
If $f\in\ker(\dB)$ is slice-regular, then $f$ is locally constant. The same holds if $f\in\ker(\partial_\B)$ and it is anti-slice-regular, namely, it satisfies $\dd {f}x=0$.
\item[(h)]
If $f\in\sr(\OO)$ and $\Delta_\B f=0$, then $f$ is an affine function of the form $xa+b$, with $a,b\in A$.
\end{itemize}
\end{proposition}
\begin{proof}
Since $f$ is a slice function, the functions $\vs f$ and $\sd f$ are constant on the spheres $\SS_x\cap M$.
Therefore every operator $L_{ij}$ vanishes on them. From the formula $f(x)=\vs f(x)+\IM(x)\sd f(x)$ and the Leibniz rule we get
\[
L_{ij}f=L_{ij}(\vs f(x))+L_{ij}(\IM(x)f'_s(x))=L_{ij}(\IM(x))\sd f(x)=(x_iv_j-x_jv_i)\sd f(x).
\]
Using again formulas \eqref{eq:iji},\eqref{eq:ijj} and point (a) of Proposition \ref{pro:gamma} we get 
\begin{align*}
\Gamma_\B f&=-\tfrac12\sum_{i,j}v_i(v_j((x_iv_j-x_jv_i)\sd f(x)))=-\tfrac12\sum_{i,j}v_i(v_j((x_iv_j-x_jv_i)))\sd f(x)\\
&=\Gamma_\B(x)\sd f(x)=(m-1)\IM(x)\sd f(x)
\end{align*}
and point (a) is proved. 
Point (b) follows immediately from (a) and Theorem \ref{teo:difference}, since a slice function $f$ is slice-regular if and only if $\difbar f=0$ on $\OO\setminus\rr$ and $\difbar f=\dd f{x^c}$ for every slice function (see \cite[Theorem 2.2]{Gh_Pe_GlobDiff}). 
To prove point (c), we observe that on $\OO\setminus\R$ it holds $\dif f=\dd fx$ and
\[
\partial_\B f-\dif f=(\partial_\B+\dB)f-\dB f-(\dif+\difbar)f+\difbar f=\partial_0f-\dB f-\partial_0f+\difbar f=
-\cS f'_s,
\]
using again Theorem \ref{teo:difference} and point (a) above. The last statement in (c) comes from the property  $(f'_s)'_s=0$, which holds for every slice function $f$. 
Points (d) and (e) are consequences of (b).

From (c) and (d) we get the first equality in \eqref{eq:Delta}:
\[\textstyle\dd{}x{(\sd f)}=\dif(\sd f)=\cS^{-1}\partial_\B(\dB f)=(4\cS)^{-1}\Delta_\B f\]
for every slice-regular $f$. Finally, since $f$ is slice-regular, the function $\tilde f=\vs f-\IM(x)\sd f$ is anti-slice-regular, namely, it satisfies $\dd {\tilde f}x=0$. Therefore 
\[
\dd f x=\dd{(\vs f+\IM(x)\sd f)}x=2\dd{(\IM(x)\sd f)}x=\dd{((x-x^c)\sd f)}x=\sd f+2\IM(x)\dd{(\sd f)}x,
\]
from which also the second equality in \eqref{eq:Delta} follows.

From (b) we get that if $\dB f=\dd f {x^c}=0$, then $\sd f\equiv0$. This means that the component $F_2$ of the inducing stem function $F$ vanishes identically. From the holomorphicity of $F$ it follows that $F_1$ is locally constant, and then also $f$ is locally constant. If instead $\partial_\B f=\dd f {x}=0$, then using (c) we conclude in the same way. Then (g) is proved.

Finally, to get (h) we follow the proof given in \cite[Lemma 23(a)]{EigenvaluePbms} in the quaternionic case.
From point (f), we get that $\dd{(\sd f)}{x}=0$, i.e., $\sd f=\I(\beta^{-1}F_2(\alpha+i\beta))$ is anti-slice-regular on $\OO$. Then the function $\beta^{-1}F_2(\alpha+i\beta)$ is constant, i.e., $F_2=\beta a$ with $a\in A$. Since $F$ is holomorphic, it follows that $F(z)=\alpha a+b+i\beta a=za+b$, and $f=\I(F)=xa+b$, with $b\in A$.
\end{proof}

\begin{remark}
The previous result shows that the operator $\Gamma_\B$, when acting on slice functions, depends only on $S$ (equivalently, on $M$) and not on the particular basis $\B$ chosen to represent it. 
%The same is true for the Cauchy-Riemann operator $\dB$,  since Theorem \ref{teo:difference} gives $\dB =\difbar +\frac12\tfrac{\IM(x)}{n(\IM(x))} \Gamma_\B$. \marginpar{true?}
\end{remark}

Let $\OO_D$ be an axially symmetric open subset of $\Q_A$ and let $\OO:=\OO_D\cap M$. 
Assume $m=\dim(M)-1>1$, i.e., that $M$ is an hypercomplex subspace of $A$.
Given a slice-regular function $f:\OO\to A$, also the pointwise product $xf$ is slice-regular on $\OO$. 
It follows that  $f=\vs f+\IM(x)\sd f=\sd{(xf)}-x^c\sd f(x)=h_1-x^ch_2$, with functions $h_1:=\sd{(xf)}$ and $h_2:=\sd f$ that are constant on every sphere $\s_x\cap M$ ($x\in\OO$). 
For this reason, we call the equality $f=h_1-x^ch_2$ the \emph{zonal decomposition} of $f$.
Since $f$ is real analytic \cite[Prop.\ 7]{AIM2011}, also $h_1$ and $h_2$ are real analytic. 

From Proposition \ref{pro:gamma_slice}(b) we get the following result.  

\begin{corollary}\label{cor:zonal_decomposition}
Let $f\in\sr(\OO)$, with $\OO$ an axially symmetric open subset of $M$.
If $A$ is not associative, assume that $f$ and $xf$ are $M$-admissible. Then the functions $h_1, h_2$ in the zonal decomposition $f=h_1-x^ch_2$ can be computed by means of the Cauchy-Riemann operator of $M$:
\begin{equation}\label{eq:h1h2}
h_1=\cS^{-1}\dB(xf),\quad h_2=\cS^{-1}\dB f,
\end{equation}
where $\cS=(1-m)/2$, $m=\dim(M)-1>1$. 
Moreover, $f$ is slice-preserving if and only if $h_1$ and $h_2$ are real-valued. \qed
\end{corollary}

\begin{remark}
When $A=M=\HH$, the functions $h_1,h_2$ in the zonal decomposition are harmonic in the four real variables. This is the content of the Almansi-type theorem proved in \cite{AlmansiH}. If $A$ is the real Clifford algebra $\R_n$, with $n\ge3$ odd, then the functions  $h_1,h_2$ are polyharmonic of degree $(n-1)/2$, i.e., in the kernel of the power $\Delta^{(n-1)/2}$ of the Laplacian of the paravector space $M=\R^{n+1}$ (see \cite{Almansi2019}). In Section \ref{sec:Laplacian} we will generalize these results to any hypercomplex subspace $M$.
\end{remark}

\section{Strongly slice-regular and locally slice-regular functions}\label{sec:strong_local}

Let $\OO$ be an open subset of an hypercomplex subspace $M$ of $A$, not necessarily axially symmetric. We recall that $\dim(M)>2$ by definition. Let $S=M\cap\s_A$ be the gis associated to $M$. 
We give a refinement of the definition of slice regularity suggested by Corollary \ref{cor:zonal_decomposition}.

\begin{definition}\label{def:ssr}
Let $\OO\subseteq M$ be any open set and let $f\in\mathcal C^1(\OO)$. The function $f$ is called \emph{strongly slice-regular} (and we write $f\in\sr^+(\OO)$) if the following properties hold:
\begin{enumerate}
\item[(i)] For every $I\in S$ such that $\OO_I:=\OO\cap \cc_I\not=\emptyset$, the restriction 
\[f_I:=f_{|\OO_I}:(\OO_I,L_I)\to(A,L_I)\]
 of $f$ is holomorphic w.r.t.\ the complex structure $L_I$ defined by left multiplication by $I$.
\item[(ii)] The functions $\dB f$ and $\dB(xf)$ are \emph{zonal}, i.e., they are constant on $\s_y\cap \OO$ for every $y\in \OO$.
\end{enumerate}
If property (i) holds and the functions $\dB f$ and $\dB(xf)$ in property (ii) are only locally constant on $\s_y\cap \OO$ for every $y\in \OO$, then $f$ is called \emph{locally slice-regular}, and we write $f\in\sr_{loc}(\OO)$.
\end{definition}

Thanks to the orthonormality of the bases, condition (ii) and then also the definitions of strongly or locally slice-regularity are independent from the choice of $\B$.
%\marginpar{defns are independent of $\B$}

\begin{remark}
If a function $f\in\mathcal C^2(\OO)$ is locally slice-regular then it satisfies on $\OO$ the system of differential equations
\begin{equation}\label{eq:system}
\begin{cases}
\difbar f%=\difbar(xf)
=0,\\
L_{ij}\dB f=L_{ij}\dB (xf)=0\text{\quad for every $i,j$ with $1\le i,j\le m$}.\\
\end{cases}
\end{equation}
\end{remark}
In a subsequent work we will investigate on the sufficiency of equations \eqref{eq:system} for the local slice-regularity.

It always holds $\sr^+(\OO)\subseteq\sr_{loc}(\OO)$. Under axial symmetry conditions, the two function spaces can coincide.

\begin{proposition}\label{pro:axially}
Let $\OO$ be an axially symmetric domain of $M$ and let $f\in\sr(\OO)$. If $A$ is not associative, assume that $f$ and $xf$ are $M$-admissible. Then $f\in\sr^+(\Omega)$. In particular, if $A$ is associative or $M=\oo$, then 
$\sr^+(\Omega)=\sr_{loc}(\OO)=\sr(\Omega)$. %for every axially symmetric domain $\OO$ in $M$. 
\end{proposition}
\begin{proof}
The first statement follows from Corollary \ref{cor:zonal_decomposition}. If $A$ is associative, or if the functions $f$ and $xf$ are $M$-admissible for every $f\in\sr(\OO)$, then we get the inclusion $\sr(\OO)\subseteq\sr^+(\OO)$. Since $\OO$ is axially symmetric, the sets $\s_y\cap\OO$ are connected for every $y\in\OO$. Therefore it holds also $\sr_{loc}(\OO)=\sr^+(\OO)$.
\end{proof}

We now show that strong slice regularity is exactly the condition which assures that the function can be extended slice-regularly to an axially symmetric set. We recall the definition of the \emph{symmetric completion} of any subset $\OO$ of $A$:
\[\widetilde \OO=\bigcup_{x\in \OO}\s_x.\]
It is the smallest axially symmetric subset of $A$ containing $\OO$.

\begin{theorem}[Global extendibility of strongly slice-regular functions]\label{teo:extension}
Let $\OO\subseteq M$ be open and let $f\in\mathcal C^1(\OO)$. If $A$ is not associative, we assume that $f$ and $xf$ are $M$-admissible. Then the function $f$ is strongly slice-regular on $\OO$ if and only if it can be extended to a slice-regular function on the symmetric completion $\widetilde \OO$. If this is the case, the extension is unique.
%, i.e., $\sr^+(U)\simeq\sr(\widetilde U)$.
\end{theorem}
\begin{proof}
Assume that $f$ is strongly slice-regular on $\OO$. Then $\difbar f=\difbar(xf)=0$ on $\OO\setminus\R$, since the restrictions $f_I$ and $(xf)_I$ are holomorphic on $(\OO_I, L_I)$.  From Theorem \ref{teo:difference} and Proposition \ref{pro:gamma}(b) we get on $\OO\setminus\R$
\begin{align*}
\dB f&=-(2\IM(x))^{-1}\Gamma_\B f \quad\text{and}\\
\dB (xf)&=-(2\IM(x))^{-1}\Gamma_\B (xf)=
\tfrac{1-m}2f(x)-(2\IM(x))^{-1}x^c\,\Gamma_\B f.
\end{align*}
Since $\IM(x)x^c=x^c\IM(x)$, we obtain $\dB(xf)=\cS f(x)+x^c\dB f$, which is equivalent to the equality $f(x)=h_1(x)-x^c h_2(x)$ for every $x\in\OO\setminus\R$, where $h_1:=\cS^{-1}\dB(xf)$ and $h_2:=\cS^{-1}\dB f$. The functions $h_1,h_2$ are  continuous on $\OO$, and they are zonal since $f$ is strongly slice-regular. 
%thanks to condition $\textrm{(ii)}$ of Def.\ \eqref{def:ssr}. 
%functions on $\OO$ defined as in \eqref{eq:h1h2}. 
The zonal decomposition 
\begin{equation}\label{eq:zonal_decomposition}
f=h_1-x^c h_2
\end{equation}
holds on the whole set $\OO$ by continuity. 

We now show that $h_1$ and $h_2$ are real analytic on $\OO\setminus\R$. 
Let %$\Phi:\C\times\s_A\to \Q_A$
$\Phi:\C\times S\to M$ be the map defined by $\Phi(\alpha+i\beta,K):=\alpha+K\beta$. 
Given $y=\tilde x_0+ I\tilde\beta\in \OO\setminus\R$, let $W:=\Phi(V\times V')\subset \OO\setminus\R$ be the neighbourhood of $y$ obtained by an open neighbourhood $V$ of $\tilde x_0+i\tilde\beta$ in $\C\setminus\R$ and an open neighbourhood $V'$ of $I$ in $S$. Let $J\in V'\setminus\{I\}$ be fixed. 
Let $x=x_0+K\beta\in W$ and consider the points $x_I:=x_0+I\beta$, $x_J:=x_0+J\beta$ in $W$. 
From the equalities
\begin{equation}\label{eq:fI}
f_I(x_I)=h_1(x)-x^c_Ih_2(x), \quad f_J(x_J)=h_1(x)-x^c_Jh_2(x)
\end{equation}
we deduce
\[
h_2(x)=(x^c_J-x^c_I)^{-1}\left(f_I(x_I)-f_J(x_J)\right)=(I-J)^{-1}\beta^{-1}\left(f_I(x_I)-f_J(x_J)\right).
\]
Since $f_I$ and $f_J$ are holomorphic, the preceding formula shows that $h_2$ is real analytic on $W$, and therefore on $\OO\setminus\R$. Since from \eqref{eq:fI} we have $h_1(x)=f_I(x_I)+x^c_I h_2(x)$, also $h_1$ and $f$ are real analytic on $\OO\setminus\R$. 

Let $\widetilde{h_1}$ and $\widetilde{h_2}$ be the functions defined on the symmetric completion $\widetilde{\OO}$ obtained by imposing their constancy on the spheres $\s_y$ for each $y\in \OO$, and set $\widetilde f(x):=\widetilde{h_1}(x)-x^c\widetilde{h_2}(x)$. From \eqref{eq:zonal_decomposition}, it follows that $\widetilde f=f$ on $\OO$. The functions $\widetilde{h_1}$ and $\widetilde{h_2}$ are real analytic on $\widetilde \OO\setminus\R$. This can be obtained as before, from the formula  
\[
\widetilde{h_2}(x)=(x^c_J-x^c_I)^{-1}\left(f_I(x_I)-f_J(x_J)\right)=(I-J)^{-1}\beta^{-1}\left(f_I(x_I)-f_J(x_J)\right),
\]
valid for every $x=x_0+K\beta\in \widetilde W=\Phi(V\times \s_A)$, %or V\times S ??
where $I$ and $J$ are fixed units chosen as above.  
Since $\widetilde{\OO}\cap\R=\OO\cap\R$, $\widetilde{h_1}$ and $\widetilde{h_2}$ are continuous on $\widetilde{\OO}$. 
From the formula $\widetilde{h_1}(x)=f_I(x_I)+x^c_I \widetilde{h_2}(x)$, we get that also $\widetilde{h_1}$ and $\widetilde f$ are real analytic on $\widetilde \OO\setminus\R$ and continuous on $\widetilde{\OO}$. 

Let $D:=\{x_0+i\beta\in \C\;|\;\text{there exists }I\in S \text{ such that }x_0+I\beta\in \OO\}$. Then $\Omega_D=\widetilde{\OO}$.
In view of the axial symmetry of $\widetilde{h_1}$ and $\widetilde{h_2}$, given $x=x_0+I\beta\in \widetilde{\OO}$ and $z=x_0+i\beta\in D$, the function
\[G(z)=G_1(z)+iG_2(z),\]
with $G_1(z):=\widetilde{h_1}(x)-x_0\widetilde{h_2}(x)$ and $G_2(z)=\beta\widetilde{h_2}(x)$, 
is a continuous stem function on $D$ that induces $\widetilde f$. If $x\in\widetilde{\OO}\cap\R$, then $\I(G)(x)=G_1(x_0)=h_1(x_0)-x_0h_2(x_0)=f(x)$, while if $x\in\widetilde{\OO}\setminus\R$, it holds
\[
\I(G)(x)=G_1(z)+\frac{\IM(x)}{\|\IM(x)\|}G_2(z)=\widetilde{h_1}(x)-x_0\widetilde{h_2}(x)+\IM(x)\widetilde{h_2}(x)=\widetilde f(x).
\] 

We now prove that the slice function $\widetilde f$ is slice-regular on $\OO_D=\widetilde \OO$. Since $\difbar\widetilde f=\difbar f=0$ on $\OO\setminus\R$, from the real analyticity of $\difbar\widetilde f$ and the fact that every connected component of $\widetilde \OO\setminus\R$ intersects at least one connected component of $\OO\setminus\R$, it follows that $\difbar\widetilde f=0$ on $\widetilde \OO\setminus\R$. This means that $(\widetilde f)_I$ is holomorphic on $\widetilde \OO_I\setminus\R$ for every $I\in\s_\hh$. In view of the continuity of $\widetilde f$, $(\widetilde f)_I$ is holomorphic on $\widetilde \OO_I$ for all $I$ (by %Morera's Theorem 
PainlevÃ©'s theorem applied to the $\C_I$-components of $(\widetilde f)_I$), and then $\widetilde f\in\sr(\widetilde \OO)$. In particular, $\widetilde f$ is real analytic on the whole set $\widetilde \OO$.

Conversely, if $f\in\mathcal C^1(\OO)$ has an extension $\widetilde f\in\sr(\widetilde \OO)$, then Corollary \ref{cor:zonal_decomposition} implies that $f$ is strongly slice-regular on $\OO$. 

The uniqueness of the extension is immediate from the identity principle for slice-regular functions on axially symmetric domains (see \cite[Thm.\ 4.11]{AlgebraSliceFunctions}).
\end{proof}

\begin{remark}
If $f\in\sr^+(\OO)$, then Proposition \ref{pro:gamma_slice}(b) gives $\dB(xf)=\cS(\vs{(\widetilde f)}(x)+x_0 {(\widetilde f)}'_s(x))$ and $\dB f=\cS(\widetilde f)'_s(x)$ for every $x\in \OO$, where $\widetilde f$ is the extension of Theorem \ref{teo:extension}. 
\end{remark}

\begin{remark}\label{rem:localnotstrong}
If $\OO$ is not axially symmetric in $M$, it can happen that a locally slice-regular function does not extend slice-regularly to $\widetilde \OO$ and therefore $\sr_{loc}(\OO)\ne\sr^+(\OO)$ (for an example in the quaternionic case $A=M=\hh$, see \cite{DouRenI,DouRenSabadini} and \cite[Example 3.4]{GS2020geometric}). 
If the intersections $\s_y\cap \OO$ are connected for every $y\in\OO$, then $\sr_{loc}(\OO)=\sr^+(\OO)$ by continuity of the functions $\dB f$ and $\dB(xf)$. For example, this is true for every open ball in $M$ w.r.t.\ the norm $\|\ \|$.
\end{remark}

\begin{theorem}[Local extendibility of locally slice-regular functions]\label{thm:localext}
Let $\OO\subseteq M$ be open and let $f\in\mathcal C^1(\OO)$. If $A$ is not associative, we assume that $f$ and $xf$ are $M$-admissible. Then $f\in\sr_{loc}(\OO)$ if and only if for each $y\in \OO$ there exists an open neighbourhood $\OO_y\subseteq\OO$ of $y$ and a slice-regular function $g_y\in\sr(\widetilde{\OO_y})$ that extends the restriction $f_{|\OO_y}$.
\end{theorem}
\begin{proof}
Let $f\in\sr_{loc}(\OO)$. If $y\in \OO\setminus\R$, we can find an open neighbourhood $\OO_y\subseteq\OO$ of $y$ such that $f_{|\OO_y}\in\sr^+(\OO_y)$ and apply Theorem \ref{teo:extension}. For example, with the same notations used in the proof of Theorem \ref{teo:extension}, if $y=\tilde x_0+ I\tilde\beta\in \OO\setminus\R$ we can take $\OO_y:=\Phi(V\times V')\subset \OO\setminus\R$ with $V$ open neighbourhood of $\tilde x_0+i\tilde\beta$ in $\cc\setminus\R$ and $V'$ open connected neighbourhood of $I$ in $S$. Then for every $x=x_0+K\beta\in \OO_y$, the set $\s_x\cap \OO_y=\Phi(\{x_0+i\beta\}\times V')$ is connected. If $y\in \OO\cap\R$, then any open ball $\OO_y\subseteq\OO$ centred in $y$ is axially symmetric, and therefore $f_{|\OO_y}\in\sr(\OO_y)=\sr(\widetilde{\OO_y})$. The converse is immediate. 
\end{proof}

If $f\in\sr_{loc}(\OO)$ and $y\in\OO$, then $\dB f(x)=\cS\sd{(g_y)}(x)$ for all $x\in\OO_y$ (Proposition \ref{pro:gamma_slice}(b)), where $g_y$ is the local extension at $y$ of Theorem \ref{thm:localext}. 
Since $g_y$ is real analytic \cite[Prop.\ 7]{AIM2011}, also $\dB f$ is real analytic on $\OO$. 
Since $\dim(M)>2$ and $c_m<0$, we can then extend to every locally slice-regular function the definitions of spherical derivative and degenerate set.

\begin{definition}\label{def:sd}
The \emph{spherical derivative} of a locally slice-regular function $f\in\mathcal C^1(\OO)$ is the real analytic function $\sd f:=c_m^{-1}\,\dB f$on $\OO$. Its zero set $D_f:=\{x\in\OO\,|\, \dB f(x)=0\}$ is called the \emph{degenerate set} of $f$. 
\end{definition}

Note that when $\OO$ is axially symmetric and $f\in\sr(\OO)$, the spherical derivative is defined, a priori, on $\OO\setminus\R$ \cite[Def.\ 6]{AIM2011}. 
Definition \ref{def:sd} gives the continuous extension of $\sd f$ to the whole $\OO$ and $D_f$ is the closure in $\OO$ of the zero set of $\sd f$.  

\begin{remark}
In the quaternionic case, with $A=M=\hh$, it has been proved  \cite{GS2020local} that if $\OO$ is a \emph{slice domain} (i.e., $\OO\cap\R\ne\emptyset$ and every slice $\OO\cap\cc_I$ is a domain in $\cc_I$), then every function $f\in\mathcal C^1(\OO)$ satisfying condition (i) in Definition \ref{def:ssr} (i.e., \emph{slice regular} on $U$ in the sense of the original definition \cite{GeSt2007Adv}), has the local slice-regular extension property, i.e., it is locally slice-regular. 
In the same paper \cite{GS2020local} it has been identified a large class of domains, the one of \emph{simple domains}, such that $f$ is also strongly slice-regular. This class includes, for example, every convex slice domain.
Observe that every locally slice-regular function is also a \emph{locally slice function} in the sense of \cite[Def.\ 3.6]{GS2020geometric}.

The function introduced in \cite[Example 2.5]{DouRenI} (see also \cite[Example 2.10]{GS2020geometric}) has a local but not a global slice-regular extension to the symmetric completion. In view of Theorems \ref{teo:extension} and \ref{thm:localext}, that function is locally slice-regular but not strongly slice-regular.
\end{remark}

\begin{remark}
If one assumes $\dim(M)=2$, then $M=\cc_I$ for some $I\in\s_A$ and the definition of local slice-regularity given in Definition \ref{def:ssr} reduces to $L_I$-holomorphy.  Every holomorphic function on $\OO\subseteq M$ satisfies the local extendibility property (and not necessarily the global one if %$\OO\cap\rr\ne\emptyset$ and 
$\OO$ is not axially symmetric, i.e., $\OO\ne\OO^c:=\{z\in M\,|\,z^c\in\OO\}$). In this sense, in the setting of slice analysis on hypercomplex subspaces, local slice-regularity can be seen as the most faithful generalization of the concept of holomorphy. %Note that when $\dim(M)=2$, the definition of local slice-regularity given in Definition \ref{def:ssr} reduces to holomorphicity.
\end{remark}

% \begin{proposition}
% Let $\dim(M)=2$. Let $\OO\subseteq M$ be open and let $f$ be holomorphic on $\OO$. %If $A$ is not associative, we assume that $f$ and $xf$ are $M$-admissible. 
% Then for each $y\in \OO$ there exists an open neighbourhood $\OO_y\subseteq\OO$ of $y$ and a holomorphic function $g$ on the symmetric completion $\widetilde{\OO_y}=\OO_y\cup\Conj{(\OO_y})$ that extends the restriction $f_{|\OO_y}$. If $\OO$ is not axially symmetric and $\OO\cap\rr\ne\emptyset$, there exists a holomorphic function $h$ on $\OO$ which has no holomorphic extension to $\widetilde{\OO}=\OO\cup\Conj{(\OO)}$.
% \end{proposition}

\subsection{Quaternionic local slice analysis}\label{sec:quatlocal}

In this section we review some results of local slice analysis in its original setting, when the hypercomplex subspace $M$ is the whole skew-field $\hh$ of quaternions. 
Let $\B=\{1,i,j,k\}$ be the standard basis of $\hh$. Then the operator $\dB$ is the Cauchy-Riemann-Fueter operator. 
By \emph{local slice analysis} we mean the set of properties satisfied by local slice-regular functions. Thanks to the Local Extendibility Theorem (Theorem \ref{thm:localext}), every \emph{local} property satisfied by slice-regular functions, originally proved on axially symmetric domains, remains valid for locally slice-regular functions defined on any open subset of $\hh$. Of course, this principle is not restricted to the quaternionic setting, but can be rephrased over any hypercomplex subspace.

In particular, given any open set $\OO\subseteq\hh$, without any assumptions about axial symmetry or non-empty intersection with the real axis, we can state that every quaternionic $f\in\sr_{loc}(\OO)$ is real analytic on $\OO$ and it is sense-preserving, i.e., its Jacobian determinant $\det (J_f)$ is never negative (see \cite{JGEA2021}). 
Moreover, the \emph{singular set}  
\[
N_f:=\{x\in\OO\,|\, \det(J_f(x))=0\}
\]
of $f$ is a real analytic subset of $\OO$. 

Another result one can obtain is the Quasi-open Mapping Theorem (see \cite[Theorem\ 6.5]{JGEA2021} for the case of strongly slice-regular functions). We recall that a continuous map $g$ between topological spaces $X$ and $Y$ is called \emph{quasi-open} if, for each $y\in g(X)$ and for each open set $U$ in $X$ that contains a compact connected component of $g^{-1}(y)$, $y$ is in the interior of $g(U)$. If $g$ is quasi-open and each of its fibers has a compact component then $g(X)$ is open in $Y$ 
%. The map $g$ is called \emph{light} if, for each $y\in Y$, the fiber $g^{-1}(y)$ is totally disconnected. If $g$ is light and quasi-open, then $g$ is open  
(see e.g.\ \cite{TitusYoung}).

We will denote by $\SC_{loc}(\OO)$ the set of functions $f\in\sr_{loc}(\OO)$ such that $\dif f=0$ on $\OO\setminus\R$. Given $y\in\OO$, Theorem \ref{thm:localext} gives a neighbourhood $\OO_y$ of $y$ and  $g_y\in\sr(\widetilde{\OO_y})$ that extends $f_{|\OO_y}$.
Then $f\in\SC_{loc}(\OO)$ if and only if for every $y\in\OO$ it holds $\dd{g_y}{x^c}=0$ and $\dd{g_y}x=0$ on $\widetilde{\OO_y}$, i.e., $g_y$ is slice-constant on $\widetilde{\OO_y}$.

Given $f\in\sr_{loc}(\OO)$ not locally constant and $y\in\OO$, let $\OO_y$ and  $g_y\in\sr(\widetilde{\OO_y})$ be as above. 
We will denote by $W_f$ the set of all the \emph{wings} of $f$, i.e., the union 
\[
W_f=\cup_{y\in\OO}(W_{g_y}\cap\OO_y)
\]
of the family of wings of $g_y$ in $\OO_y$. These are real analytic submanifolds of $\OO$ of dimension 2 contained in particular fibers $f^{-1}(c)\supseteq g_y^{-1}(c)\cap\OO_y$ of $f$, for some $c\in\hh$. If $\OO_y\cap\R\ne\emptyset$, then $W_{g_y}=\emptyset$ (see \cite[\S5]{JGEA2021} for definition and properties of the wings of a slice-regular function). Observe that when $\OO$ is a slice domain, the wing set $W_f$ is empty. This follows from the properties of the zero sets of slice-regular functions on slice domains \cite[Cor.\ 5.3]{GS2020geometric}.

\begin{theorem}[Quasi-open Mapping Theorem]\label{thm:quasiopen}
Let $f\in\sr_{loc}(\OO)\setminus\SC_{loc}(\OO)$. Then
\begin{enumerate}
 \item $f$ is quasi-open.
  \item The restriction $f|_{\OO\setminus \left(D_f\cup W_f\right)}$ is open.
% \item The restriction $f|_{\OO\setminus \left(\overline{D_f}\cup W_f\right)}$ is open.
\end{enumerate}
\end{theorem}
\begin{proof}
Since $f$ is not slice-constant, the real analytic set $N_f$ has dimension less then four. This follows from \cite[Theorem\ 6.4]{JGEA2021} applied to the local extensions of $f$ provided by Theorem \ref{thm:localext}. 
Since the Jacobian does not change sign on $\OO$, it follows from results of Titus and Young \cite{TitusYoung} that $f$ is quasi-open.

For any $y\in\OO$, let $\OO_y$ and  $g_y\in\sr(\widetilde{\OO_y})$ be as in the Local Extendibility Theorem \ref{thm:localext}. It holds $D_f\cap\OO_y=D_{g_y}\cap\OO_y$ and $W_f\cap\OO_y=W_{g_y}\cap\OO_y$. The first equality is immediate from definitions. The second one follows from the characterization of the wings of $g_y\in\sr(\widetilde{\OO_y})$: $g_y$ has a wing $W_{g_y,c}\subseteq W_{g_y}$ if and only if $N(g_y-c)\equiv0$ \cite[Cor.\ 5.3]{JGEA2021}. Therefore, if $x\in W_f\cap\OO_y$, then there exists $z\in\OO$ such that $x\in W_{g_z}\cap\OO_z$, with $g_z\equiv f\equiv g_y$ on the intersection $\OO_z\cap\OO_y\ni x$. Then $N(g_y-c)=N(g_z-c)\equiv0$ on $\OO_z\cap\OO_y$, and $x\in W_{g_y}\cap\OO_y$.
The other inclusion $W_f\cap\OO_y\supseteq W_{g_y}\cap\OO_y$ is obvious. 

If $U$ is an open subset of $\OO\setminus(D_f\cup W_f)$, then $U\cap\OO_y$ is an open subset of $\OO_y\setminus(D_{g_y}\cup W_{g_y})$ and in view of \cite[Theorem\ 6.5]{JGEA2021} $g_y$ is an open map when restricted to $\widetilde{\OO_y}\setminus(D_{g_y}\cup W_{g_y})$. Therefore 
\[
f(U)=\cup_{y\in\OO}f(U\cap\OO_y)
\]
is open in $\hh$. This proves that $f|_{\OO\setminus \left(D_f\cup W_f\right)}$ is open.
\end{proof}

In the case of slice domains, the Open Mapping Theorem (point (2) of the previous Theorem) was proved in \cite[Theorem\ 10.5]{GS2020geometric}. 

% \begin{enumerate}
%  \item If $\OO$ is a slice domain, then $f(\OO)$ is open in $\hh$ and the restriction $f|_{\OO\setminus \overline{D_f}}$ is open.
%  \item If $\OO$ is a product domain, then the restriction $f|_{\OO\setminus \left(D_f\cup W_f\right)}$ is open. Moreover, if $W_f=\emptyset$, then $f(\OO)$ is open in $\hh$.
% \end{enumerate}

% ($f\in\mathcal C^\omega$, $\det J_f\ge0$, the Quasi-open Mapping Theorem (use $\dim N_f<4$ for $f$ not slice-constant, see proof of \cite[Prop.\ 10.4]{GS2020geometric}), the Open-Mapping Theorem (discrete fibers outside $D_f\cup W_f$))

Two further results of local slice analysis are related to the Almansi-type decomposition obtained in \cite{AlmansiH}. 
Let $\mathbb B$ be the open unit ball in $\R^4$ and let $\s^3=\partial\mathbb B$ be the three-dimensional unit sphere.  Let $\sigma$  denote the normalized rotation-invariant surface-area measure on the unit sphere $\SS^3$ of $\hh\simeq\R^4$, such that $\sigma(\SS^3)=1$.  

\begin{corollary}[Mean value formula% for locally slice-regular functions
]\label{cor:mvf_strongly}
Let $f\in\sr_{loc}(\OO)$. Assume that the open ball  $B(a,r)$ with centre $a\in\OO$ and radius $r$ has closure $\overline{B(a,r)}$ contained in $\OO$.  Then
\begin{equation}\label{eq:MVF}
f(a)=\int_{\SS^3}f(a+r\zeta)d\sigma(\zeta)-r \int_{\SS^3}\overline\zeta\, \dcf f(a+r\zeta)d\sigma(\zeta).
\end{equation}
\end{corollary}
\begin{proof}
Let $r'>r$ such that $\overline{B(a,r)}\subset B(a,r')\subset \OO$.
In view of Remark \ref{rem:localnotstrong}, it holds $f\in\sr_{loc}(B(a,r'))=\sr^+(B(a,r'))$. Let $\widetilde f$ be  the slice-regular extension of $f_{|B(a,r')}$ on $\widetilde{B(a,r')}$ provided by Theorem \ref{teo:extension}. Since it holds $\sd {(\widetilde f)}=-\dcf\widetilde f=-\dcf f$ on $B(a,r')$, formula \eqref{eq:MVF} follows from \cite[Prop.\ 2]{AlmansiH} applied to $\widetilde f$.
\end{proof}

Let $P(x,\zeta)=(1-|x|^2)/|x-\zeta|^4$ be the Poisson kernel of the unit ball $\BB$ in $\R^4$.

\begin{corollary}[Poisson formula% for locally slice-regular functions
]\label{cor:Poisson_strongly}
Let $f\in\srl(\OO)$. Assume that the open ball  $B(a,r)$ with centre $a\in \OO$ and radius $r$ has closure $\overline{B(a,r)}$ contained in $\OO$.  Then, for every $x\in\BB$, it holds
\[f(a+rx)=\int_{\SS^3}f(a+r\zeta)P(x,\zeta)d\sigma(\zeta)-r \int_{\SS^3}(\overline\zeta-\overline x)\,\dcf f(a+r\zeta) P(x,\zeta)d\sigma(\zeta).
\]
\end{corollary}

\begin{proof}
We can argue as in the preceding proof. The formula follows from \cite[Prop.\ 3]{AlmansiH} applied to the extension $\widetilde f\in\sr(\widetilde{B(a,r')})$.
\end{proof}

When the centre $a$ is a real point, the ball $B(a,r)$ is an axially symmetric set. In this case we can obtain two formulas in which %the spherical derivative of $f$ 
$\dcf f$ does not appear, as in \cite[Cor.\ 2]{AlmansiH}.

%\section{Laplacian operators on hypercomplex subspaces}\label{sec:Laplacian}
\section{Polyharmonicity properties of slice-regular functions}\label{sec:Laplacian}

\subsection{The action of the Laplacian on slice-regular functions}

Let $\OO_D$ be an axially symmetric open subset of $\Q_A$ and $M$ an hypercomplex subspace of $A$.
Let $\OO=\OO_D\cap M$ be an axially symmetric open subset of $M$. %Assume that $m=\dim(M)-1$ is greater than 1.
Let $f=\I(F)$ be a slice-regular function on $\OO_D$, with $F=F_1+\ui F_2$ a holomorphic stem function with components $F_1$, $F_2$. The functions $F_1$ and $F_2$ have harmonic real components with respect to the two-dimensional Laplacian $\Delta_2$ of the plane. 
Since $F(\overline z)=\overline{F(z)}$ for every $z=\alpha+i\beta$, the functions $F_1,F_2:D\subseteq\C\to A$ are, respectively, even and odd functions with respect to the variable $\beta$. Therefore there exist real analytic functions $G_1,G_2:\widetilde{D}\to A$ such that for every $\alpha+i\beta\in D\setminus\R$, it holds
 \begin{equation}\label{eq:G1G2}
F_1(\alpha, \beta)= G_1(\alpha,\beta^2),\quad F_2(\alpha, \beta)=\beta G_2(\alpha,\beta^2),
\end{equation}
where $\widetilde{D}=\{z=\alpha+i\beta^2\in\C\,|\,\alpha+i\beta\in D,\,\beta>0\}$.
If $x=\alpha+\beta J\in\OO_D\setminus\rr$, $z=\alpha+i\beta\in D\setminus\R$, then 
\begin{align}
%\vs f(x)&=G_1(\alpha,\beta^2)=G_1(\RE(x),\|\IM(x)\|^2),\label{g1}
%\\
f'_s(x)&=G_2(\alpha,\beta^2)=G_2(\RE(x),\|\IM(x)\|^2).\label{g2}
\end{align}
In the following, the symbol  $\partial_uG_j(u,v)$ stands for the partial derivative $\ddd{G_j}u(u,v)$ and $\partial_vG_j(u,v)$ for the partial derivative $\ddd{G_j}v(u,v)$ for $j=1,2$.
The functions $G_1$ and $G_2$ are useful in the computation of the Laplacian of the spherical derivative %and of the spherical value 
of a slice regular function. Here we generalize results proved in \cite{Harmonicity} in the setting of Clifford algebras.

\begin{proposition}\label{pro:laplacian}
Let $\OO=\OO_D\cap M$ be an axially symmetric open subset of $M$.
Let $f=\I(F):\OO\to A$ be slice-regular. Let $f'_s(x)=G_2(\RE(x),\|\IM(x)\|^2)$ as in \eqref{g2} and $m=\dim(M)-1$. Then it holds:
\begin{itemize}\setlength\itemsep{0.1em}
\item[(a)]
For every $x\in\OO\setminus\rr$, 
\[\Delta_{\B}\sd f(x)=2(m-3)\,\partial_v G_2(\RE(x),\|\IM(x)\|^2).\]
\item[(b)]
For each $k=1,2,\ldots, \left[\frac{m-1}2\right]$ and every $x\in\OO\setminus\rr$, 
\[\Delta_{\B}^k \sd f(x)=2^k(m-3)(m-5)\cdots(m-2k-1)\,\partial_v^k G_2(\RE(x),\|\IM(x)\|^2).\]
% \item[(c)]
% \[\Delta_{n+1}f'_s(x)=\frac{n-3}{|\IM(x)|^2}\left(\dd{\vs f}{x_0}(x)-f'_s(x)\right).\]
\end{itemize}
%More precisely, the formulas in (a) and (b) hold if and only if $F_2$ has harmonic components on $D$.
In particular, the functions $\Delta_{\B}^k \sd f(x)$ do not depend on the basis $\B$ of $M$.
\end{proposition}
\begin{proof}
%The computations are similar to those made in %\cite{Qian1997} and 
%\cite[Theorem 11.33]{GHS}. 
We can follow the lines of the proof given in \cite[Theorem 4.1]{Harmonicity}.
Let $x_0=\RE(x)$, $r=\|\IM(x)\|$. By direct computation, from \eqref{eq:G1G2} and \eqref{g2} we get
\begin{equation}\label{eq:Delta2}
\Delta_2F_2(\alpha,\beta)=\beta\left(\dds u^2+4\beta^2\,\dds v^2 +6\dds v \right)G_2(\alpha,\beta^2)
\end{equation}
and
\begin{align}\label{eq:DeltaG2}
\Delta_{\B}G_2(x_0,r^2)&=
\partial_0^2G_2(x_0,r^2)+\sum_{i=1}^m\partial_i^2G_2(x_0,r^2)=\left(\dds u^2 +4r^2\dds v^2 +2m\,\dds v\right) G_2(x_0,r^2).
%\frac{\partial^2{G_2}}{\partial x_0^2}(x_0,r^2)+\sum_{i=1}^m\frac{\partial^2{G_2}}{\partial x_i^2}(x_0,r^2)=\left(\dds 1^2 +4r^2\dds 2^2 +2n\,\dds 2\right) G_2(x_0,r^2).
\end{align}
Since $\Delta_2F_2=0$ on $D$, we get $\Delta_{\B}\sd f(x)=\Delta_{\B}G_2(x_0,r^2)=(2m-6)\,\dds v G_2(x_0,r^2)$. This proves (a). To obtain (b) we use induction on $k$, starting from the case $k=1$ given by (a). For every $k$ with $1<k\le\left[\frac{m-1}2\right]-1$, using equation \eqref{eq:DeltaG2} with the function $\dds v^k G_2(x_0,r^2)$ in place of $G_2(x_0,r^2)$ and again $\Delta_2F_2=0$ we get
\begin{align*}
\Delta_{\B}\dds v^k G_2(x_0,r^2)&%=\dds v^k\Delta_{\B} G_2(x_0,r^2)
=\left(\dds u^2\dds v^k +4r^2\dds v^{k+2} +2m\,\dds v^{k+1}\right) G_2(x_0,r^2)\\
&=\left(\dds v^k\left(\dds u^2 +4r^2\dds v^2 +2m\,\dds v\right)-4k\,\dds v^{k+1}\right) G_2(x_0,r^2)\\
&=\left(\dds v^k\left(-6\,\dds v +2m\,\dds v\right)-4k\dds v^{k+1}\right) G_2(x_0,r^2)\\
&=2(m-2k-3)\,\dds v^{k+1}  G_2(x_0,r^2).
\end{align*}
By the induction hypothesis 
\begin{align*}
\Delta_{\B}^{k+1}\sd f(x)&=2^k(m-3)(m-5)\cdots(m-2k-1)\,\Delta_{\B}\partial_v^k G_2(x_0,r^2)\\&=
2^k(m-3)(m-5)\cdots(m-2k-1)\, 2(m-2k-3)\,\dds v^{k+1} G_2(x_0,r^2)
\end{align*}
and (b) is proved.
% Statement (c) follows from the holomorphicity of $F$. Since $\dds\alpha F_1(\alpha,\beta)=\dds\beta F_2(\alpha,\beta)=\dds\beta (\beta G_2(\alpha,\beta^2))=G_2(\alpha,\beta^2)+2\beta^2\dds 2 G_2(\alpha,\beta^2)$, it holds, for $r\ne0$,
% \[
% \dds 2 G_2(x_0,r^2)=\frac1{2r^2}\left(\dds 1 F_1(x_0,r^2)-G_2(x_0,r^2)\right)=\frac1{2r^2}\left(\dd {\vs f}{x_0} (x)-\sd f (x)\right).
% \]
% Together with (a), this proves (c).
\end{proof}

Using the previous computations, we are now able to prove some harmonicity properties of slice-regular functions. In particular, we  give a version of the Fueter-Sce-Qian Theorem (see e.g.\ \cite[\S1.1.3]{AltavillaDeBieWutzig} and references therein) for the hypercomplex subspace $M$ of $A$.  In the case of Clifford algebras $\rr_n$ with $n$ odd and $M=\rr^{n+1}$, these results were proven in \cite[Corollaries 4.2 and 4.4]{Harmonicity}. The case of Clifford algebras $\rr_n$ with $n$ even was proved in \cite[Lemma 4.4]{AltavillaDeBieWutzig} by means of the definition of the fractional Laplacian via Fourier transform, the approach already used by Qian in \cite{Qian1997}. Here we adopt a different definition of the fractional Laplacian $(-\Delta)^{1/2}$, the one based on the \emph{harmonic extension problem} introduced by Caffarelli and Silvestre \cite{CaffarelliSilvestre} on $\rr^N$  and generalized to open subsets $U$ of $\rr^N$ in \cite{StingaTorrea}. We briefly recall this approach. Under suitable regularity conditions for $g:U\to\R$, if the function $u(x,y):U\times\rr^+\to\R$ solves the problem
\[
\begin{cases}
\Delta_x u(x,y)+\dd{^2u}{y^2}(x,y)=0\text{\quad on }U\times\rr^+,\\
u(x,0)=g(x)\text{\quad on }U,\\
\dd{u}{y}(x,0)=h(x)\text{\quad on }U,
\end{cases}
\]
then it holds $(-\Delta)^{1/2}g=h$ on $U$.

\begin{theorem}\label{thm:n_odd_even}
Let $d=\dim(A)$ and let $f:\OO\subseteq M\to A$ be slice-regular.  

If $m=\dim(M)-1\ge3$ is odd, then it holds:
\begin{itemize}\setlength\itemsep{0.4em}
\item[(a)]
  If $m=3$, $\Delta_\B\sd f=0$, i.e.,  the $d$ real components of the spherical derivative $\sd f$ are $\Delta_\B$-harmonic on $\OO\setminus\rr$.
\item[(b)]
If $m>3$, $(\Delta_{\B})^{\frac{m-1}2}\sd f=0$, i.e., $\sd f$ is polyharmonic of order $\frac{m-1}2$ on $\OO\setminus\rr$.
\end{itemize}
If $A$ is not associative, assume also that $f$ is $M$-admissible. It holds:
\begin{itemize}\setlength\itemsep{0.4em}
\item[(c)]
%The following version of Fueter-Sce Theorem holds for $M$:
$\dB(\Delta_{\B})^{\frac{m-1}2}f=(\Delta_{\B})^{\frac{m-1}2} \dB f=0$
on $\OO$.
\item[(d)]
 $(\Delta_{\B})^{\frac{m+1}2} f=0$ on $\OO$, i.e., $f$ %every slice-regular function on $\OO$ 
 is polyharmonic of order $\frac{m+1}2$ on $\OO$.
\end{itemize}

If $m=\dim(M)-1\ge2$ is even, then it holds:
\begin{itemize}\setlength\itemsep{0.4em}
\item[(a')]
If $m=2$, $(-\Delta_\B)^{1/2}\sd f=0$, i.e.,  the $d$ real components of the spherical derivative $\sd f$ are in the kernel of the fractional Laplacian $(-\Delta_\B)^{1/2}$ on $\OO\setminus\rr$.
\item[(b')]
If $m>2$, $(-\Delta_\B)^{1/2}((\Delta_{\B})^{\frac{m-2}2}\sd f)=0$ on $\OO\setminus\rr$.
\end{itemize}
If $A$ is not associative, assume also that $f$ is $M$-admissible. It holds:
\begin{itemize}\setlength\itemsep{0.4em}
\item[(c')]
$(-\Delta_\B)^{1/2}(\Delta_{\B})^{\frac{m-2}2} \dB f=0$
on $\OO$.
\item[(d')]
 $(-\Delta_\B)^{1/2}(\Delta_{\B})^{\frac{m}2} f=0$ on $\OO$.
\end{itemize}
\end{theorem}
\begin{proof}
%Since $f\in\sr(\OO)$, $\sd f$ extends continuously to $\OO$ (see \cite[\S3.3]{AIM2011}). 
Assume $m$ odd.
Points (a) and (b) are immediate from Proposition \ref{pro:laplacian}. Point (c) is a consequence of Proposition \ref{pro:laplacian} and Proposition \ref{pro:gamma_slice}(b). Statement (d) comes from the factorization $\Delta_\B=4\partial_\B\dB$ and point (c). Since $\cS\sd f=\dB f$ on $\OO\setminus\rr$, the spherical derivative $\sd f$ extends real analytically to $\OO$. Therefore (c) and (d) hold on the whole $\OO$.

% Now assume $m$ even and let $M$ be induced by a gis $S=M\cap\s_A$. We claim that there exists an imaginary unit $v_{m+1}\in\s_A\setminus S$, and then such that $v_{m+1}\not\in M$. If the products $v_1v_j$ belong to $M$ for every $j=2,\ldots,m$, then $M$ would be a real vector space with a complex structure given by left multiplication by $v_1$. But this is not compatible with the odd dimension of $M$. Let $v_{m+1}:=v_1v_j\not\in M$. Since $t(v_{m+1})=0$ by \eqref{eq:t2} and $n(v_{m+1})=(v_1v_j)(v_jv_1)=1$, it holds $v_{m+1}\in\s_A$ and this proves the claim. 
% Let $\B':=\{1,v_1,\ldots,v_m,v_{m+1}\}$ and $M':=\Span(\B')$.

Now assume $m\ge2$ even. Since $\s_A\ne\emptyset$ and every $J\in\s_A$ defines a complex structure on $A$ by left multiplication by $J$, the dimension $d$ of $A$ is even. Therefore $M\ne A$. Let $\B'=\B\cup\{v_{m+1}\}=\{1,v_1,\ldots,v_m,v_{m+1}\}$ be an orthonormal set w.r.t.\ the norm $\|\ \|$ in $A$ and let $M':=\Span(\B')\simeq\rr^{m+2}$. 
Notice that $M'$ is not necessarily contained in the quadratic cone $\Q_A$.
Let $\Delta_{\B'}$ be the Laplacian operator on $M'$ associated with $\B'$. If we denote by $x_0,x_1,\ldots,x_m,x_{m+1}$ the coordinates of $M'$ w.r.t.\ $\B'$, given an open subset $\OO'$ of $M'$, $\Delta_{\B'}$ acts on functions of class $C^2(\OO',A)$ as
\[
\Delta_{\B'}h:=\Delta_\B h+\partial^2_{m+1}h
\]
where the operator $\partial^2_{m+1}:C^2(\OO',A)\to C^0(\OO',A)$ is defined by 
$\partial^2_{m+1} h=L\circ\dd{^2(L^{-1}\circ h\circ L)}{x_{m+1}^2}\circ L^{-1}$, as in Definition \ref{def:db}.

Let $f'_s(x)=G_2(\alpha,\beta^2)=G_2(\RE(x),\|\IM(x)\|^2)$ for $x\in\OO\setminus\rr$ as in \eqref{g2}. 
From Proposition \ref{pro:laplacian}(b), it follows that 
\[g(x):=\Delta_{\B}^{(m-2)/2} \sd f(x)=
a_m\,\partial_v^{{(m-2)/2}} G_2(x_0,r^2),
\]
where $r^2=\sum_{i=1}^mx_i^2$ and $a_m:=2^{(m-2)/2}(m-3)!!$ for $m\ge4$, $a_2:=1$. Let $\tilde r:=\sqrt{r^2+x_{m+1}^2}$.
The function $g$ can be extended smoothly for $x=\sum_{i=0}^{m+1}x_iv_i\in\OO'=\OO\times\Span(v_{m+1})\subset M'$ as
\[\tilde g(x):=a_m\,\partial_v^{{(m-2)/2}} G_2(x_0,\tilde r^2)=a_m\,\partial_v^{{(m-2)/2}} G_2(x_0, r^2+x_{m+1}^2).
\]
Then it holds
\[
\partial_{m+1}\tilde g(x)=2a_mx_{m+1}\partial_v^{m/2} G_2(x_0,\tilde r^2)
\]
and
\[
\partial^2_{m+1}\tilde g(x)=2a_m\left(\partial_v^{m/2} G_2(x_0,\tilde r^2)+
2x_{m+1}^2\partial_v^{(m+2)/2} G_2(x_0,\tilde r^2)\right).
\]
As seen in equation \eqref{eq:Delta2}, $\Delta_2F_2(\alpha,\beta)=0$ on $D$ implies that
\[
\left(\dds u^2+4\beta^2\,\dds v^2 +6\dds v \right)G_2(\alpha,\beta^2)=0\text{\quad for $\beta\ne0$}.
\]
Therefore when $m=2$,
\begin{align*}
\Delta_{\B'}\tilde g&=\left(\dds u^2+4 r^2\,\dds v^2 +4\dds v \right)G_2(x_0,\tilde r^2)+2\left(\partial_v G_2(x_0,\tilde r^2)+
2x_{3}^2\partial_v^2 G_2(x_0,\tilde r^2)\right)\\
&=
\left(\dds u^2+4\tilde r^2\,\dds v^2 +6\dds v \right)G_2(x_0,\tilde r^2)=0\text{\quad for $\tilde r\ne0$},
\end{align*}
and when $m\ge4$,
\begin{align*}
\Delta_{\B'}\tilde g&
=a_m\left(\dds u^2\dds v^{{(m-2)/2}} +4r^2\dds v^{(m+2)/2} +2m\,\dds v^{m/2}\right) G_2(x_0,\tilde r^2)\\
&\quad +2a_m\left(\partial_v^{m/2} G_2(x_0,\tilde r^2)+
2x_{m+1}^2\partial_v^{(m+2)/2} G_2(x_0,\tilde r^2)\right)\\
&=a_m\left(\dds v^{(m-2)/2}\left(\dds u^2 +4\tilde r^2\dds v^2-4x_{m+1}^2\dds v^2 +2m\,\dds v\right)-2(m-2)\,\dds v^{m/2}\right) G_2(x_0,\tilde r^2)\\
&\quad +2a_m\left(\partial_v^{m/2} G_2(x_0,\tilde r^2)+
2x_{m+1}^2\partial_v^{(m+2)/2} G_2(x_0,\tilde r^2)\right)\\
&=a_m\left(\dds v^{(m-2)/2}\left(-6\partial_v-4x_{m+1}^2\dds v^2 +2m\,\dds v\right)-2(m-2)\,\dds v^{m/2}\right) G_2(x_0,\tilde r^2)\\
&\quad +2a_m\left(\partial_v^{m/2} G_2(x_0,\tilde r^2)+
2x_{m+1}^2\partial_v^{(m+2)/2} G_2(x_0,\tilde r^2)\right)=0\text{\quad for $\tilde r\ne0$}.
\end{align*}
Therefore, for every $m$ it holds:
\[
\begin{cases}
\Delta_{\B'}\tilde g=0\text{\quad on }\OO'\setminus\rr,\\
\partial_{m+1}\tilde g=0\text{\quad when }x_{m+1}=0.
\end{cases}
\]
This means that $(-\Delta)^{1/2}g=0$ on $\OO\setminus\rr$ and proves  points (a') and (b'). The equality of statement (c') on $\OO\setminus\rr$  is a consequence of (a'), (b') and Proposition \ref{pro:gamma_slice}(b). 
Since the operators $\dB$, $\Delta_\B$ and $(-\Delta_\B)^{1/2}$ commute with each other, statement (d') on $\OO\setminus\rr$ follows from the factorization $\Delta_\B=4\partial_\B\dB$ and point (c'). By continuity, (c') and (d') hold on the whole $\OO$.
\end{proof}

Given any $n\in\qq$ such that $2n\in\nn$, we will say that an $A$-valued function $f$ is \emph{$n$-polyharmonic} if it belongs to the kernel of the operator $(\Delta_\B)^{n}$ when $n\in\nn$, and to the kernel of the operator $(-\Delta_\B)^{1/2}(\Delta_\B)^{\frac{2n-1}2}$ when $2n$ is odd. Theorem \ref{thm:n_odd_even} shows that for any slice-regular function $f$ on $M$, the spherical derivative $\sd f$ is $\frac{m-1}2$-polyharmonic, while $f$ is $\frac{m+1}2$-polyharmonic, where $m=\dim(M)-1$. In view of Theorem \ref{thm:localext}, the same holds for any locally slice-regular function on $M$. 

%\marginpar{$Im(\mathcal F)$ of Fueter mapping [Sce]: slice derivatives of (slice) zonal harmonics? nec.\ $f$ real?}

\begin{corollary}
Let $f$ be a locally slice-regular function on an open set $\OO\subseteq M$. If $A$ is not associative, assume also that $f$ is $M$-admissible. Then $f$ is $\frac{m+1}2$-polyharmonic, where $m=\dim(M)-1$. 
\end{corollary}

Table 1 summarizes the particular cases when $M$ is one of the three real division algebras $\cc,\hh$ or $\oo$. These are the unique cases where $M$ can taken to be the whole algebra $A$ (\cite[Prop.\ 1(7)]{AIM2011}). In the table the symbol $\Delta_{2k}$ denotes the Euclidean Laplacian of $\R^{2k}$ for $k=1,2,4$. We recall that quaternionic and octonionic slice-regular functions are sense-preserving, i.e., their Jacobian determinant is always non-negative (see \cite{JGEA2021} and \cite{OCSoctonions}), as it holds for every complex holomorphic map.

%\vskip 8pt %\hskip-10pt%\noindent 

\begin{table}[h]
\begin{center}
% \begin{tabular}{c}
% \\
% $f$ slice-regular\\
% \end{tabular}
$f$ locally slice-regular$\quad\Rightarrow\quad$
{
\begin{tabular}{c|c|c}
%\hline
\multicolumn{3}{c}{Real division algebras} \\
\hline
$\cc$ &$\hh$&$\oo$ \\
\hline
$\Delta_2^1f=0$ & $\Delta_4^2f=0$ & $\Delta_8^4f=0$\\[2pt]
\hline
harmonic&biharmonic&4-harmonic\\
\multicolumn{3}{c}{(sense-preserving)}
%\hline
\end{tabular}
\vskip 8pt
}
\caption{$M=\cc,\hh$ or $\oo$.}
\end{center}
\end{table}

As an application of Theorem \ref{thm:n_odd_even}, we can refine Corollary \ref{cor:zonal_decomposition} and generalize to every hypercomplex subspace the Almansi type decompositions obtained in \cite{AlmansiH,Almansi2019} in the setting of quaternions and Clifford algebras.

\begin{corollary}[Polyharmonic Zonal Decomposition]\label{cor:PZDec}
Let $f\in\sr(\OO)$, with $\OO$ an axially symmetric open subset of $M$.
If $A$ is not associative, assume that $f$ and $xf$ are $M$-admissible. Let $m=\dim(M)-1\ge2$. Then there exist two uniquely determined $A$-valued zonal and $\frac{m-1}2$-polyharmonic functions $h_1$, $h_2$ on $\OO$ such that
\[
f(x)=h_1(x)-x^c h_2(x)
\]
for every $x\in\OO$. The same result holds locally for every locally slice-regular function. The function $f$ is slice-preserving if and only if $h_1$ and $h_2$ are real-valued.
 \qed
\end{corollary}

\subsection{The $\dibar\Delta$-decomposition and the $\dibar\Delta$-Fueter mapping}\label{sec:DDec}

Let $M$ be an hypercomplex subspace of $A$ of dimension $m+1$.
When $m=3$, any slice-regular function on subsets of $M$ is biharmonic. %and the decomposition of Corollary \ref{cor:PZDec} contains zonal harmonic slice functions $h_1$, $h_2$. 
When $m$ is odd and greater than 3, we can combine Corollary \ref{cor:PZDec} with the classical Almansi's Theorem  on polyharmonic functions (see \cite{Almansi,Aronszajn}) to obtain a decomposition of any slice-regular function in terms of biharmonic slice functions, more precisely with components in the kernel of the third order operator $\dB\Delta_\B$. 

% When $m$ is odd, we can iteratively apply Proposition \ref{pro:AMDec} and obtain a decomposition of any slice-regular function in terms of polymonogenic slice functions, i.e., with components in the kernel of the iterated Cauchy-Riemann operator $(\dB)^{\frac{m-1}2}$. When $m=3$, the result reduces to Proposition \ref{pro:AMDec}. 

\begin{theorem}[$\dibar\Delta$-decomposition]\label{thm:dDeltaDec}
Let $f\in\sr(\OO)$, with $\OO$ an axially symmetric open subset of $M$.  Assume that $\OO=\OO_D\cap M$ is a star-like domain with centre 0.
If $A$ is not associative, assume that $f$ and $xf$ are $M$-admissible. Let $m=\dim(M)-1\ge3$ be odd. 
Then there exist $A$-valued %(zonal?? $m=3$?) 
slice functions $f_0,\ldots,f_{(m-3)/2}$ on $\OO$ in the kernel of the operator $\dB\Delta_{\B}$ such that%\marginpar{unique?} 
\[
f(x)=f_0(x)+\|x\|^2f_1(x)+\cdots+\|x\|^{m-3}f_{\frac{m-3}2}(x)\text{\quad  $\forall x\in\OO$.}
\]
In particular, the functions $f_0,\ldots,f_{(m-3)/2}$ are biharmonic. The function $f$ is slice-preserving if and only if $f_0,\ldots,f_{(m-3)/2}$ are slice-preserving.
\end{theorem}
\begin{proof}
Let $f=h_1-x^ch_2$ be the decomposition given in Corollary~\ref{cor:PZDec}. We recall that $h_1=\sd{(xf)}$ and $h_2=\sd f$.
We identify as usual $M$ with $\R^{m+1}$ by means of the coordinates $x_0,\ldots,x_m$ w.r.t.\ the basis $\B$. Then for every $x\in M$, $\|x\|$ coincides with the Euclidean norm of $\R^{m+1}$.

Let $(e_0,e_1,\ldots,e_{d-1})$ be a real basis of $A$. If $h$ is any slice function on $\OO_D$, the decomposition $h(x)=\sum_{i=0}^{d-1}h^i(x)e_i$ defines $d$ real-valued functions $h^0,\ldots,h^{d-1}$. We show that the $h^i$ are slice functions on $\OO_D$, using the characterization of sliceness given in \cite[Lemma\ 3.2]{Gh_Pe_GlobDiff}. Since $h$ is slice, the functions 
\[
h(x)+h(x^c)=\sum_{i=0}^{d-1}(h^i(x)+h^i(x^c))e_i\quad\text{and}\quad 
\IM(x)(h(x)-h(x^c))=\sum_{i=0}^{d-1}\IM(x)(h^i(x)-h^i(x^c))e_i,
\]
%where $J_x=\IM(x)/|\IM(x)|$ for $x\in\OO_D\setminus\rr$, 
are constant on a fixed sphere $\s_y$. Therefore also the real component functions 
\[
h^i(x)+h^i(x^c)\quad\text{and}\quad \IM(x)(h^i(x)-h^i(x^c)),\quad i=0,\ldots, d-1,
\]
are constant on $\s_y$. This implies that every $h^i$ is a slice function on $\OO_D$. Applying this argument to $h_1$ and $h_2$, we get that the $d$ real components of $h_1$ and $h_2$ are slice functions on $\OO_D$. Now we apply Almansi's Theorem to these real components, which are $\frac{m-1}2$-polyharmonic on $\OO$. We get $A$-valued harmonic functions $u_0,\ldots,u_{(m-3)/2}$ and $v_0,\ldots,v_{(m-3)/2}$ such that 
\[
h_1(x)=\sum_{k=0}^{(m-3)/2}\|x\|^{2k}u_k(x)\quad\text{and}\quad h_2(x)=\sum_{k=0}^{(m-3)/2}\|x\|^{2k}v_k(x).
\]
The proof of Almansi's Theorem given, e.g., in \cite[Prop.\ 1.3]{Aronszajn}, shows that also the functions $u_k$, $v_k$ are slice functions on $\OO$. This follows from the fact that given a slice function $g:\OO\to\R$, the function defined for an integer $k\ge1$ by
\begin{equation}\label{eq:aronszajn}
x\mapsto\int_0^1\tau^{k-2+(m+1)/2}g(\tau x)d\tau
\end{equation}
is again slice on $\OO$. This can be seen using again the sliceness criterion recalled above. Formula \eqref{eq:aronszajn} is the basic step in the iterative procedure used in the construction of the harmonic components in the Almansi decomposition.

Let $f_k:=u_k-x^c v_k$ for $k=0,\ldots, (m-3)/2$. Then the $f_k$'s are slice functions on $\OO$, such that
\[
f(x)=h_1(x)-x^ch_2(x)=\sum_{k=0}^{(m-3)/2}\|x\|^{2k}(u_k(x)-x^c v_k(x))=\sum_{k=0}^{(m-3)/2}\|x\|^{2k}f_k(x).
\]
% The functions $f_k=u_k-x^c v_k$ have the same symmetry properties as the functions $A$ and $B$. This follows from the uniqueness of the Almansi decomposition. Given any orthogonal transformation $T$ of $\R^{n+1}$ that fixes the real points, also $A\circ T$ is $m$-harmonic and $u_k\circ T$ is harmonic. Since
% \[
% A(T(x))=\sum_{k=0}^{m-1}|x|^{2k}u_k(T(x))=A(x),
% \] 
% it must be $u_k(T(x))=u_k(x)$ for every $k$. The same holds for $B$. 
To prove the last statement of the thesis, we observe that for every $A$-valued function $u$, a direct computation gives $\Delta_{\B}(x^c u)=4\partial_\B u+x^c\Delta_{\B}u$. Therefore 
$\Delta_{\B}f_k=\Delta_{\B}(-x^c v_k)=-4\partial_B v_k$ and then $\dibar_\B\Delta_{\B}f_k=-4\dibar_\B\partial_\B v_k=-4\Delta_{\B} v_k=0$. In particular, the slice functions $f_k$ are biharmonic: $\Delta_\B\Delta_B f_k=4\partial_\B(\dB\Delta_\B f_k)=0$.

If $f$ is slice-preserving, then $h_1$ and $h_2$ are real-valued (Corollary \ref{cor:PZDec}). Therefore also the $u_k$'s and $v_k$'s are real-valued and $f_k=u_k-x^cv_k$ sends any slice $\OO\cap \cc_J$ into $\cc_J$, i.e., it is slice-preserving. Conversely, if $f_k(\OO\cap \cc_J)\subseteq\cc_J$ for every $k$ and every $J\in S$, then the same holds for $f=\sum_{k}\|x\|^{2k}f_k$. 
\end{proof}

\begin{remark}
If $\OO$ contains the closure of the unit ball $B:=B_0(1)$ in $M$, and $f=\sum_{k=0}^{(m-3)/2}\|x\|^{2k}f_k$ is the $\dibar\Delta$-decomposition of $f\in\sr(\OO)$, then the $A$-valued slice function
\[
g:=f_0+f_1+\cdots+f_\frac{m-3}2
\]
is a solution of the boundary value problem
\begin{equation}
\begin{cases}\label{BVP1}
\dB\Delta_\B g=0\quad\text{on $B$,}\\
g_{|\s^m}=f_{|\s^m},
\end{cases}
\end{equation}
where $\s^m=\partial B$ is the $m$-dimensional unit sphere in $M$. If $v:=v_0+v_1+\cdots+v_{(m-3)/2}$, with the harmonic functions $v_k$ defined as in the proof of Theorem \ref{thm:dDeltaDec}, then $\Delta_\B f_k=-4\partial_\B v_k$ for every $k$, and therefore $\Delta_\B g=-4\partial_\B v$. This means that the pair of slice functions $(g,v)$ is the unique solution of the boundary value problem
\begin{equation}
\begin{cases}\label{BVP2}
\Delta_\B g+4\partial_\B v=0\quad\text{on $B$,}\\
\Delta_\B v=0\quad\text{on $B$,}\\
g_{|\s^m}=f_{|\s^m},\\
v_{|\s^m}=(\sd f)_{|\s^m}.
\end{cases}
\end{equation}
Uniqueness for the solution of \eqref{BVP2} can seen as follows. If $f\equiv0$, then \eqref{BVP2} implies that $v\equiv0$ since it is harmonic on $B$ and vanishing on $\s^m$. Then also $g$ is harmonic and therefore $g\equiv0$. Note that if $(g,v)$ solves \eqref{BVP2}, then $g$ solves also \eqref{BVP1}.
\end{remark}

Given $f\in\sr(\OO)$, the procedure described in the proof of the preceding theorem gives a unique $\frac{m-1}2$-tuple $(f_0,\ldots,f_{(m-3)/2})$ of slice functions. This follows from the uniqueness of the pair $(h_1,h_2)$ in Corollary \ref{cor:PZDec} and the uniqueness part of Almansi's Theorem. In the following we will refer to these functions $f_0,\ldots,f_{(m-3)/2}$ as the components of the \emph{$\dibar\Delta$-decomposition} of $f$. 
More precisely, the correspondence $f\mapsto (f_0,\ldots,f_{(m-3)/2})$ defines an injective $\rr$-linear map 
\[
\Tdd:\sr(\OO)\to\left(\ker(\dB\Delta_\B)\right)^{\tfrac{m-1}2}.
\]
If we compose $\Tdd$ with the operator $\Delta_\B$ acting on every components of $(f_0,\ldots,f_{(m-3)/2})$, we get a $\rr$-linear map (we call it the \emph{$\dibar\Delta$-Fueter mapping})
%\marginpar{$\Hdd:\sr(\OO)\to\left(\AH(\OO)\right)^{\tfrac{m-1}2}$\\$\Hdd(f)=(\dB f_i)_i$ ''axially harmonic''}
\[
\Fdd:\sr(\OO)\to\left(\AM(\OO)\right)^{\tfrac{m-1}2},
\]
where $\AM(\OO)$ is the class of \emph{axially monogenic functions}, i.e., of slice functions on $\OO$ in the kernel of $\dB$:
\[\AM(\OO)=\{f\in\SL^1(\OO)\,|\,\dB f=0\}.
\]

When $m=3$, the previous theorem reduces to point (c) of Theorem\ \ref{thm:n_odd_even} (in this case $\Tdd$ is simply the inclusion operator and $\Fdd=\Delta_\B$). This means that Theorem \ref{thm:dDeltaDec} can be considered as a different generalization of the quaternionic Fueter Theorem, valid on any hypercomplex subspace $M$ of $A$. 
When $M$ is the paravector subspace $\R^{m+1}$ of $\R_m$, then the Sce's generalization \cite{Sce} of Fueter's Theorem  provides a mapping $\Delta^{(m-1)/2}$ from $\sr(\OO)$ to $\AM(\OO)$ (usually called \emph{Fueter mapping}). %One possible drawback of this result is that 
This operator 
%The operator $\Delta^{(m-1)/2}$ defined on $\sr(\OO)$ 
has as a ever larger kernel as $m$ increases. On the contrary, we show that the $\dibar\Delta$-Fueter mapping $\Fdd$ has a small kernel for every $m\ge3$.
In a subsequent paper we will study a characterization of the image of the $\Fdd$-Fueter mapping in $\left(\AM(\OO)\right)^{\frac{m-1}2}$.

\begin{proposition}\label{pro:Kernel}
Let $f\in\sr(\OO)$, with $\OO$ an axially symmetric domain of $M$.  Assume that $\OO=\OO_D\cap M$ is a star-like domain with centre 0.
If $A$ is not associative, assume that $f$ and $xf$ are $M$-admissible. Let $m=\dim(M)-1\ge3$ be odd.
If $f\in\ker(\Fdd)$, then $f$ is an affine function, namely, there exist $a,b\in A$ such that $f(x)=xa+b$.
\end{proposition}
\begin{proof}
We adopt the same notation used in the proof of Theorem \ref{thm:dDeltaDec}. 
The functions $u_k$, $v_k$ have the same axial symmetry properties as the functions $h_1=\sd{(xf)}$, $h_2=\sd f$, which are zonal with pole 1. This follows from the uniqueness of the Almansi decomposition. Given any orthogonal transformation $T$ of $M\simeq\R^{m+1}$ that fixes the real points, also $h_1\circ T$ is $\frac{m-1}2$-polyharmonic and $u_k\circ T$ is harmonic. Since
\[
h_1(T(x))=\sum_{k=0}^{(m-3)/2}\|x\|^{2k}u_k(T(x))=h_1(x),
\] 
it must be $u_k(T(x))=u_k(x)$ for every $k$. Similarly for $h_2$ and the $v_k$'s. It follows that $v_k=\sd{(f_k)}$, since
\begin{align*}
\sd{(f_k)}&=(x-x^c)^{-1}\left(f_k(x)-f_k(x^c)\right)=(x-x^c)^{-1}\left(u_k(x)-x^cv_k(x)-u_k(x^c)+xv_k(x^c)\right)\\
&=v_k(x).
\end{align*}
If $\Fdd(f)=0$, then $\Delta_\B f_k=0$ for every $k=0,\ldots,(m-3)/2$. Since $f_k=u_k-x^cv_k$, %with harmonic $u_k$, $v_k$, 
it holds $0=\Delta_\B f_k=-4\partial_B v_k$. 
From points (c) of Proposition \ref{pro:gamma_slice} we deduce that $\dd{v_k}x=-c_m((f_k)'_s)'_s=0$ and from point (g) of the same Proposition we infer that $v_k$ must be a constant $a_k\in A$ on $\OO$. 
Therefore 
\[
\sd f(x)=h_2(x)=\textstyle\sum_{k=0}^{(m-3)/2}\|x\|^{2k}a_k=\I\left(\sum_{k=0}^{(m-3)/2}|z|^{2k}a_k\right)\text{\quad for every $x\in\OO\setminus\R$}.
\]
Let $\beta=\|\IM(x)\|$ and $z=\alpha+i\beta\in D\subset\cc$. Since $f=\I(F_1+iF_2)$ is slice-regular and $\sd f=\I(\beta^{-1}F_2)$, the function $F_2(z)=\beta\sum_{k=0}^{(m-3)/2}(\alpha^2+\beta^2)^{k}a_k$ must have harmonic real components in $D$. A direct computation shows that
\[
0=\Delta_2\textstyle F_2(z)=4\beta\left(\sum_{k=1}^{(m-3)/2}(k^2+k)|z|^{2k-2}a_k\right)
\]
from which we deduce that $a_k=0$ for every $k\ge1$. Then 
% $\sd f=h_2=a_0$ is constant, and then $\Delta_\B f=0$ from Proposition \ref{pro:gamma_slice}(f). From Proposition \ref{pro:gamma_slice}(h),  $f$ is affine.
$F_2=\beta a$ with $a=a_0\in A$. Since $F$ is holomorphic, it follows that $F(z)=\alpha a+b+i\beta a=za+b$, and $f=\I(F)=xa+b$, with $b\in A$.
\end{proof}

% \begin{corollary}
% Let $f\in\sr(\OO)$, with $\OO$ an axially symmetric open subset of $M$.  Assume that $\OO=\OO_D\cap M$ is a star-like domain with centre $y$.
% If $A$ is not associative, assume that $f$ and $xf$ are $M$-admissible. Let $m=\dim(M)-1\ge3$ be odd. 
% Then there exist $A$-valued functions $f_0,\ldots,f_{(m-3)/2}$ on $\OO$ in the kernel of the operator $\dB\Delta_{\B}$ such that\marginpar{unique?} 
% \[
% f(x)=f_0(x)+\|x-y\|^2f_1(x)+\cdots+\|x-y\|^{m-3}f_{\frac{m-3}2}(x)\text{\quad  $\forall x\in\OO$.}
% \]
% %In particular, the functions $f_0,\ldots,f_{(m-3)/2}$ are biharmonic.
% \end{corollary}
% \begin{proof}

% \end{proof}

As a first corollary of Theorem\ \ref{thm:dDeltaDec}, we obtain a local version of the $\dibar\Delta$-decomposition, valid for all local slice-regular functions.

\begin{corollary}[Local $\dibar\Delta$-decomposition]\label{cor:DDeltaDecolocal}
Let $\OO\subseteq M$ be open and let $f\in\sr_{loc}(\OO)$. If $A$ is not associative, we assume that $f$ and $xf$ are $M$-admissible. Let $m=\dim(M)-1\ge3$ be odd. For every $y\in\OO$, there exist an open ball $B_y\subseteq\OO$ centred at $y$
 and $A$-valued %locally-slice 
 functions $f_0,\ldots,f_{(m-3)/2}$  on $B_y$ in the kernel of the operator $\dB\Delta_{\B}$, %and then biharmonic, 
 such that
\[
f(x)=f_0(x)+\|x-y\|^2f_1(x)+\cdots+\|x-y\|^{m-3}f_{\frac{m-3}2}(x)\text{\quad  $\forall x\in B_y$.}
\]
\end{corollary}
%locally slice function in the sense of \cite[Def.\ 3.6]{GS2020geometric}.
\begin{proof}
By Theorem \ref{thm:localext}, there exists an open neighbourhood $\OO_y\subseteq\OO$ of $y$ and a slice-regular function $g\in\sr(\widetilde{\OO_y})$ that extends the restriction $f_{|\OO_y}$. On $\widetilde{\OO_y}$ we can write (Corollary \ref{cor:PZDec}) $g=h_1-x^ch_2$, with $h_1$ and $h_2$ $\frac{m-1}2$-polyharmonic functions.   

Let $B_y\subseteq\OO_y$ be an open ball centred at $y$. Then $B_0:=B_y-y$ is a ball centred at 0, where we can apply Almansi's Theorem to the $d$ real components of $\tilde h_1(x'):=h_1(x'+y)$ and $\tilde h_2(x'):=h_2(x'+y)$, which are $\frac{m-1}2$-polyharmonic w.r.t.\ $x'\in B_0$. We get $A$-valued harmonic functions $u_k$ and $v_k$ ($k=0,\ldots,(m-3)/2$) on $B_0$ such that 
\[\textstyle
\tilde h_1(x')=\sum_{k=0}^{(m-3)/2}\|x'\|^{2k}u_k(x')\quad\text{and}\quad \tilde h_2(x')=\sum_{k=0}^{(m-3)/2}\|x'\|^{2k}v_k(x').
\]
Then, for every $x=x'+y\in B_y$, it holds
\[\textstyle
h_1(x)=\sum_{k=0}^{(m-3)/2}\|x-y\|^{2k}u_k(x-y)\quad\text{and}\quad h_2(x)=\sum_{k=0}^{(m-3)/2}\|x-y\|^{2k}v_k(x-y).
\]
The functions $\tilde u_k(x):= u_k(x-y)$ and $\tilde v_k(x):= v_k(x-y)$ are harmonic on $B_y$. Let $f_k(x):=\tilde u_k(x)-x^c \tilde v_k(x)$ for $x\in B_y$. Then on $B_y$ we can expand $f$ as
\[\textstyle
f(x)=g(x)=h_1(x)-x^c h_2(x)=\sum_{k=0}^{(m-3)/2}\|x-y\|^{2k}f_k(x),
\]
with $\Delta_{\B}f_k=\Delta_{\B}(-x^c \tilde v_k)=-4\partial_B \tilde v_k$ and then $\dibar_\B\Delta_{\B}f_k=-4\dibar_\B\partial_\B \tilde v_k=-4\Delta_{\B} \tilde v_k=0$.
\end{proof}

Observe that the functions $f_k$'s in the expansion of Corollary \ref{cor:DDeltaDecolocal}, differently from what happens in the global case, are not necessarily slice functions or restrictions of slice functions to the ball $B_y$.

As a second corollary of Theorem \ref{thm:dDeltaDec} we give an explicit formula the $\dibar\Delta$-decomposition for slice-regular polynomials. In the case of homogeneous polynomials, the Almansi decomposition reduces to what in the literature is called the \emph{Gauss} (or \emph{canonical}) \emph{decomposition} of polynomials (see e.g.\ \cite[Ch.9]{Vilenkin}). If $p_n$ is a homogeneous polynomial of degree $n$ in $m+1$ variables $x_0,\ldots,x_m$, then there exist harmonic homogeneous polynomials $q_{n-2k}$ of degree $n-2k$ if $q_{n-2k}\not\equiv0$, for $k=0,\ldots,s$, with $s=\lfloor{n/2}\rfloor$, such that
\[
p_n=q_n+\|x\|^2q_{n-2}+\cdots+\|x\|^{2s}q_{n-2s}.
\]
%with $q_{n-2k}$ an harmonic homogeneous polynomial of degree $n-2k$ for $k=0,\ldots,s$, 
%with $s=\lfloor{n/2}\rfloor$. 
The harmonic component $\Pi_k(p_n):=q_{n-2k}$ of degree $n-2k$ of $p_n$ can be computed by means of the following formula (see e.g.\ \cite{Avery}):
\begin{equation}\label{eq:pik}
\Pi_k(p_n)=\frac{(m+2n-4k-1)!!}{(2k)!!(m+2n-2k-1)!!}\sum_{j=0}^{\lfloor{n/2-k}\rfloor} 
\frac{(-1)^j(m+2n-4k-2j-3)!!}{(2j)!!(m+2n-4k-3)!!}\|x\|^{2j}\Delta_\B^{j+k}p_n.
\end{equation}
We extend the projection operator $\Pi_k$ to any polynomial $p\in A[x_0,\ldots,x_m]$ by additivity, after writing $p$ as the sum of its homogeneous components. 

\begin{corollary}[Polynomial $\dibar\Delta$-decomposition]\label{cor:DDeltaDecPoly}
Let $f=\sum_{n=0}^Nx^na_n\in A[x]$. If $A$ is not associative, assume that $f$ and $xf$ are $M$-admissible. Let $m=\dim(M)-1\ge3$ be odd. 
Define 
%\[f_k:=\pi_k(\sd{(xf)})-x^c\pi_k(\sd f) \text{\quad for $k=0,\ldots,\tfrac{m-3}2$.}\] 
%\[f_k:=\cS^{-1}\big(\Pi_k(\dB(xf))-x^c\Pi_k(\dB f)\big) \text{\quad for $k=0,\ldots,\tfrac{m-3}2$.}\] 
\begin{equation}\label{eq:fk}
f_k:=-\tfrac{2}{m-1}\big(\Pi_k(\dB(xf))-x^c\Pi_k(\dB f)\big) \text{\quad for $k=0,\ldots,\tfrac{m-3}2$.}
\end{equation}
Then every $f_k\in A[x_0,\ldots,x_m]$ is an $A$-valued polynomial slice function in the kernel of the operator $\dB\Delta_{\B}$, such that
\[
f(x)=f_0(x)+\|x\|^2f_1(x)+\cdots+\|x\|^{m-3}f_{\frac{m-3}2}(x)\text{\quad  $\forall x\in M$.}
\]
If $\deg(f)=N$, then  $\deg(f_k)=N-2k$ or $f_k\equiv0$. The $\dibar\Delta$-Fueter mapping sends $f$ to the sequence $\Fdd(f)=(g_0,\ldots,g_{(m-3)/2})$, where 
\[
g_k:=\tfrac8{m-1}\,\partial_\B\Pi_k(\dB f)=-4\,\partial_\B\Pi_k(\sd f)
\]
is an axially monogenic polynomial %in the space $A[x_0,\ldots,x_m]$ 
of degree $N-2k-2$ (if $g_k\not\equiv0$), for $k=0,\ldots,(m-3)/2$. %in the kernel of $\dB$. 
\end{corollary}
\begin{proof}
Let $f=h_1-x^ch_2$ be the decomposition of Corollary~\ref{cor:PZDec}, with $h_1=\sd{(xf)}=c_m^{-1}\dB(xf)$ and $h_2=\sd f=c_m^{-1}\dB f$. Since $h_1$ and $h_2$ are $\frac{m-1}2$-polyharmonic, it holds $\Delta_\B^k h_j=0$ for $j=1,2$ and $k>(m-3)/2$. Then formula \eqref{eq:pik} implies that the operator $\Pi_k$ vanishes on $h_1$ and $h_2$ for every $k>(m-3)/2$. Therefore the harmonic functions $u_k=\Pi_k(h_1)$ and $v_k=\Pi_k(h_2)$ give the Almansi decomposition of $h_1$ and $h_2$:
\[\textstyle
h_1(x)=\sum_{k=0}^{(m-3)/2}\|x\|^{2k}u_k(x)\quad\text{and}\quad h_2(x)=\sum_{k=0}^{(m-3)/2}\|x\|^{2k}v_k(x).
\]
From Theorem \ref{thm:dDeltaDec}, we have $f_k=u_k-x^c v_k=\Pi_k(h_1)-x^c\,\Pi_k(h_2)$, and this concludes the proof of \eqref{eq:fk}. The last statement is a consequence of the equality $\Delta f_k=-4\partial_\B v_k=-4\partial_\B\Pi_k(h_2)$ and Proposition \ref{pro:gamma_slice}(d).
\end{proof}

\subsection{A second polyharmonic decomposition of slice-regular functions}\label{sec:different}

In the case of quaternions ($M=\hh$, $m=\dim M-1=3$), an axially monogenic decomposition for slice-regular functions was proved in \cite[Theorem 8]{NewCauchy}: a slice-regular function $f$ can be written uniquely as $f(x)=g_1(x)-\overline x g_2(x)$, where $g_1$ and $g_2$ are quaternionic axially monogenic functions. This result provides another bridge between the slice-regular function theory on one side and the one of Fueter-regular functions on the other side. 
 We now generalize this result to any hypercomplex subspace $M$ of $A$.  In the following we call a function $g\in\sr(\OO)$ a \emph{slice-regular primitive} of $f\in\sr(\OO)$ w.r.t.\ the slice derivative $\dd{}x$ if it holds $\dd g x =f$.

\begin{proposition}\label{pro:AMDec}
Let $f\in\sr(\OO)$, with $\OO$ an axially symmetric open subset of $M$.  Assume that $\OO=\OO_D\cap M$, and that every connected component of $D$ is simply connected.
If $A$ is not associative, assume that $f$ and $xf$ are $M$-admissible. Let $m=\dim(M)-1\ge2$. Then there exist two $A$-valued slice functions $g_1$, $g_2$ on $\OO$ such that 
\[
f(x)=g_1(x)-x^c g_2(x)
\]
for every $x\in\OO$, and with 
\begin{align}\label{eq:oddeven}
(-\Delta_\B)^{1/2}g_j=0&\text{\quad for $j=1,2$ if $m=2$},\notag\\
\dB(\Delta_\B)^{\frac{m-3}2}g_j=(\Delta_\B)^{\frac{m-3}2}\dB g_j=0&\text{\quad for $j=1,2$ if $m\ge3$ is odd, }\\\notag
(-\Delta_\B)^{1/2}(\Delta_\B)^{\frac{m-4}2}\dB g_j=0&\text{\quad for $j=1,2$ if $m\ge4$ is even}.
\end{align}
The functions $g_1,g_2$ can be obtained by a slice-regular primitive $g$ of $f$ by means of the formulas
\begin{equation}\label{eq:g1g2}
g_1:=(4\cS)^{-1}\Delta_\B(xg), \quad g_2=(4\cS)^{-1}\Delta_\B g.
\end{equation}
The result holds locally for every locally slice-regular function. In this case no topological assumption on $D$ is needed. 
\end{proposition}

Before proving the Proposition, we give a result about the existence of slice-regular primitives of slice-regular functions.  

\begin{lemma}
Let $\OO$ be an axially symmetric open subset of $M$. Assume that $\OO=\OO_D\cap M$, and that every connected component of $D$ is simply connected. Then every $f\in\sr(\OO)$ has a slice-regular primitive w.r.t.\ the slice derivative $\dd{}x$, namely there exists $g\in\sr(\OO)$ such that $\dd g x =f$. The primitive is unique up to the addition of a slice-constant function.
\end{lemma}
\begin{proof}
Since any axially symmetric open set is union of {}a family of slice domains or product domains, we can assume that $\OO$ is a domain of one of these types. 
We can adapt the proof given in \cite[Theorem 8]{NewCauchy} for the case of quaternions.
Let $(e_0,e_1,\ldots,e_{d-1})$ be a real basis of $A$. If $f=\I(F)$, the decomposition $F=\sum_{i=0}^{d-1}F^ie_i$ defines $d$ holomorphic stem functions $F^i:D\to\rr\otimes\cc\simeq\cc$.  

If $\OO=\OO_D$ is a slice domain, then $D$ is a simply connected subset of $\cc$. Given a holomorphic primitive $G^i:D\to\cc$ of $F^i$, for $i=0,1,\ldots,d-1$, let $\tilde G^i(z):=\frac12(G^i(z)+\overline{G^i(\overline z)})$ on $D$.
Then $\tilde G^i$ is a holomorphic stem function on $D$ such that $\dd{\tilde G^i}{z}=F^i$. The slice function $g=\I(\sum_{i=0}^{d-1}\tilde G^ie_i)$ is then a slice-regular primitive of $f$. 

If $\OO=\OO_D$ is a product domain, then $D^+:=D\cap\cc^+$ is simply connected. Then there exist holomorphic primitives $G^i_+:D^+\to\cc$ of $F^i$, for $i=0,1,\ldots,d-1$. Define $G^i_-$ on $D^-:=D\cap\cc^-$ by $G_-^i(z):=\overline{G_+^i(\overline z)}$. Then the function $G^i$ defined as $G^i_+$ on $D^+$ and as  $G^i_-$ on $D^-$ is a holomorphic stem function on $D$ with $\dd{G^i}{z}=F^i$. We conclude observing that the sum $\sum_{i=0}^{d-1} G^ie_i$ induces a slice-regular primitive of $f$. 

If $g$ and $h$ are two slice-regular primitives of $f$, then $\dd{(g-h)}x=\dd{(g-h)}{x^c}=0$. Therefore $g-h\in\SC(\OO)$,  i.e., it is induced by a locally constant stem function. 
\end{proof}

Observe that when $\OO$ is a slice domain, then a slice-constant function on $\OO$ is a constant.

An open subset $\OO_D$ such that every connected component of $D$ is simply connected is sometimes called a \emph{basic domain} in the literature (see e.g.\ \cite{GPV}).

\begin{proof}[Proof of Proposition \ref{pro:AMDec}]
Let $g$ a slice-regular primitive of $f$ on $\OO$. Then, using the Leibniz-type formula for the spherical derivative and value (see \cite[\S5]{AIM2011}) and Proposition \ref{pro:gamma_slice}(e), we get 
\begin{align*}
f=\dd g x&=\dd{}x(\vs g+\IM(x)\sd g)=\dd{}x(\vs g+(x_0-x^c)\sd g)=\dd{}x(\vs g+x_0\sd g)-x^c\dd{(\sd g)}x\\
&=\dd{(\sd{(x g)})}x-x^c\dd{(\sd g)}x= (4\cS)^{-1}\left(\Delta_\B(xg)-x^c\Delta_\B g\right).
\end{align*}
Setting
\[
g_1:=(4\cS)^{-1}\Delta_\B(xg), \quad g_2=(4\cS)^{-1}\Delta_\B g,
\]
we get $f=g_1-x^c g_2$. Since $g$ and $xg$ are slice-regular, for $m\ge3$ the validity of \eqref{eq:oddeven} follows immediately from points (c) and (c') of Theorem \ref{thm:n_odd_even}. If $m=2$, it follows from point (d') of the same Theorem.

If $f$ is only locally slice-regular, near any point $x=\alpha+J\beta\in\OO$ one can take an axially symmetric domain satisfying the topological assumption, to which $f$ extends slice-regularly. The previous argument then gives the local result.
\end{proof}

The functions $g_1$ and $g_2$ are in particular $\frac{m-1}2$-polyharmonic. Therefore Proposition \ref{pro:AMDec} refines Corollary \ref{cor:PZDec}. Observe however that now the functions are slice but not zonal in general.

\subsection{Examples: Clifford algebras, octonions, reduced quaternions}

In the quaternionic case with $M=\hh$ (and then $m=3$), the decomposition of Section \ref{sec:different} is already known (see \cite{AlmansiH,NewCauchy}) and the result proved in Theorem \ref{thm:dDeltaDec} reduces to Fueter's Theorem. We then give some examples in other hypercomplex subspaces. 

\subsubsection*{Clifford algebras $\R_m$}
Let $M=\R^{m+1}\subseteq\R_m$ be the paravector subspace of the Clifford algebra $\R_m=\R_{0,m}$ and $\B=(1,e_1,\ldots,e_m)$. Then $\dB$ is the Cauchy-Riemann operator
\[\dibar=\frac12\left(\dd{}{x_0}+e_1\dd{}{x_1}+\cdots+e_n\dd{}{x_n}\right)\]
 of $\Rn$ and $\Delta_\B$ is the Euclidean Laplacian $\Delta_{m+1}$ of $\R^{m+1}$. 

\begin{example}
Let $M=\R^6\subseteq\R_5$ be the paravector subspace of the Clifford algebra $\R_5$. %and $\B=(1,e_1,e_2,e_3,e_4,e_5)$.  Here the constant $c_m$ is equal to $-2$.  
Since $m=5$, it holds $c_m=-2$. Consider for example the power $x^3$ of the Clifford variable restricted to $\R^6$.  The biharmonic decomposition of Corollary \ref{cor:PZDec} is
\[
x^3=h_1(x)-x^c h_2(x)\text{\quad  $\forall x\in\R^6\subseteq\R_5$},
\]
with zonal biharmonics $h_1,h_2$ in $\R^6$, given by 
\begin{align*}
h_1(x)&=-\tfrac12\dibar(x^4)=4x_0(x_0^2-|\IM(x)|^2)
,\quad\text{and}\\
h_2(x)&=-\tfrac12\dibar(x^3)=3x_0^2-|\IM(x)|^2.
\end{align*}
Theorem \ref{thm:dDeltaDec} permits to refine the decomposition with functions $f_0$, $f_1$ in the kernel of the operator $\dibar\Delta$, where $\Delta=\Delta_6$. It holds $x^3=f_0+|x|^2f_1$, with
\begin{align*}
f_0(x)&=\tfrac13\left(5x_0^3-7x_0|\IM(x)|^2+10x_0^2\IM(x)-2|\IM(x)|^2\IM(x)\right)
,\\
f_1(x)&=-\tfrac13\left(2x_0+\IM(x)\right).
\end{align*}
These polynomials that can be computed applying formula \eqref{eq:pik}.
\end{example}

Proposition \ref{pro:AMDec} applied to the powers $x^k$ of the Clifford variable restricted to the paravectors, gives the next Corollary, valid for every Clifford algebra $\R_m$. 

\begin{corollary}\label{cor:xk}
For every integer $k\ge0$ and every $x\in\R^{m+1}$, it holds 
\[
x^k=-\frac1{2(m-1)(k+1)}\left(\Delta_{m+1}(x^{k+2})-x^c \Delta_{m+1}(x^{k+1})\right).
\]
The same formula holds for negative integers $k\le-2$ and $x\in\R^{m+1}\setminus\{0\}$. \qed
\end{corollary}

\begin{example}
Let again $M=\R^6\subseteq\R_5$ and $f(x)=x^3$.
The decomposition of Proposition \ref{pro:AMDec} (or Corollary \ref{cor:xk}) is 
\[
x^3=g_1(x)-x^c g_2(x),
\]
with
\begin{align*}
g_1(x)&=-\tfrac1{32}\Delta(x^5)=\tfrac52\left(x_0^3-x_0|\IM(x)|^2+x_0^2\IM(x)\right)-\tfrac12|\IM(x)|^2\IM(x)
,\quad\text{and}\\
g_2(x)&=-\tfrac1{32}\Delta(x^4)=\tfrac12\left(3x_0^2-|\IM(x)|^2+2x_0\IM(x)\right)
.
\end{align*}
It holds $\dibar g_1=-5x_0^2+|\IM(x)|^2$, $\dibar g_2=-2x_0$, and $\dibar\Delta g_j=\Delta\dibar g_j=0$ for $j=1,2$, as expected. 
\end{example}

\begin{example}\label{ex:r7}
Let $M=\R^8\subseteq\R_7$ be the paravector subspace of the Clifford algebra $\R_7$. %and $\B=(1,e_1,e_2,e_3,e_4,e_5)$.  Here the constant $c_m$ is equal to $-2$.  
It holds $c_m=-3$. Consider the slice-regular power $f(x)=x^4$ of the Clifford variable restricted to $\R^8$.  The 3-harmonic decomposition of Corollary \ref{cor:PZDec} is
% \[
% x^3=h_1(x)-x^c\,h_2(x)\text{\quad  $\forall x\in\R^8\subseteq\R_7$},
% \]
% with zonal 3-harmonics  
% \begin{align*}
% h_1(x)&=-\tfrac13\dibar(x^4)=4x_0(x_0^2-|\IM(x)|^2),\\
% h_2(x)&=-\tfrac13\dibar(x^3)=3x_0^2-|\IM(x)|^2.
% \end{align*}
% in $\R^8$. The decomposition of Theorem \ref{thm:dDeltaDec} is $x^3=f_0+|x|^2f_1+|x|^4f_2$, with $f_2\equiv0$ and
% \begin{align*}
% f_0(x)&=\tfrac1{10}\left(21x_0^3-19x_0|\IM(x)|^2+35x_0^2\IM(x)-5|\IM(x)|^2\IM(x)\right)
% ,\\
% f_1(x)&=-\tfrac1{10}\left(11x_0+5\IM(x)\right),
% \end{align*}
% such that $\dibar\Delta f_0=\dibar\Delta f_1=0$. 
\[
x^4=h_1(x)-x^c h_2(x)\text{\quad  $\forall x\in\R^8\subseteq\R_7$},
\]
with zonal 3-harmonics  
\begin{align*}
h_1(x)&=-\tfrac13\dibar(x^4)=5x_0^4-10x_0^2|\IM(x)|^2+|\IM(x)|^4,\\
h_2(x)&=-\tfrac13\dibar(x^3)=4x_0(x_0^2-|\IM(x)|^2).
\end{align*}
in $\R^8$. The decomposition of Theorem \ref{thm:dDeltaDec} is $x^4=f_0+|x|^2f_1+|x|^4f_2$, with 
\begin{align*}
f_0(x)&=\tfrac25\left(7x_0^4-12x_0^2|\IM(x)|^2+|\IM(x)|^4+14x_0^3\IM(x)-6x_0|\IM(x)|^2\IM(x)\right)
,\\
f_1(x)&=\tfrac1{10}\left(-19x_0^2+5|\IM(x)|^2-16x_0\IM(x)\right),\\
f_2(x)&=\tfrac 1{10}
\end{align*}
satisfying $\dibar\Delta f_k=0$ for $k=0,1,2$. Here $\Delta=\Delta_8$. The function $g=f_0+f_1+f_2$ solves \eqref{BVP1}, while the pair $(g,v)$, with $v=v_0+v_1+v_2=\frac45 x_0\left(7x_0^3-3|\IM(x)|^2-2\right)$, is the unique solution of the boundary problem \eqref{BVP2}. The $\dibar\Delta$-Fueter mapping $\Fdd$ sends $f$ to the triple of monogenic functions
\[
\Fdd(f)=(\Delta f_0,\Delta f_1,\Delta f_2)=\left(-\tfrac{24}5(7x_0^2-|\IM(x)|^2+2x_0\IM(x)),\,\tfrac{16}5,\,0\right).
\] 
Observe that $f$, and then also the $f_k$'s and $\Delta(f_k)$'s, are slice-preserving. In particular, all non-paravector components vanish on $M$. In the general case, these functions can have all the $2^7$ real components different from zero.
\end{example}

\subsubsection*{Octonions}% togliere?
Let $A=M=\oo$ be the algebra of octonions. We adopt here the algebra isomorphism $\oo\simeq\hh\oplus\hh$ given by the Cayley-Dickson construction, with multiplication $(q_1,q_1')(q_2,q_2')=(q_1q_2-\overline{q_2'}q_1',q_2'q_1+q_1'\overline{q}_2)$ and conjugation $(q,q')^c=(\overline{q},-q')$. 
If $x=(q_1,q_1')\in\oo$, we can take $|x|=(xx^c)^{1/2}=(|q_1|^2+|q_1'|^2)^{1/2}$ as the norm $\|x \|$ on $\oo$.

As in \cite{DentoniSce}, let $\B$ be the orthonormal basis of $\oo\simeq\R^8$ formed by elements $e_h=(i_h,0)$ for $h=0,1,2,3$ and $e_h=(0,i_{4-h})$ for $h=4,5,6,7$,  where $(i_0=1,i_1,i_2,i_3)$ is the standard basis of $\hh$. In particular, $e_0=1$ is the unity of the algebra. Then the operator $\dB$ is the octonionic Cauchy-Riemann operator (firstly introduced under the name  \emph{Fueter-Moisil operator} by Dentoni and Sce in \cite{DentoniSce}) 
\[
\dibar=\frac12\left(\dd{}{x_0}+e_1\dd{}{x_1}+\cdots +e_7\dd{}{x_7}\right),
\]
where $x_0,\ldots,x_7$ are the real coordinates w.r.t.\ $\B$, and $\Delta_\B$ is the Euclidean Laplacian $\Delta=\Delta_{8}$ of $\R^8$. 
Octonionic slice-regular functions were introduced in \cite{GeStRocky}. We refer to that paper for definitions and properties. 
Theorem \ref{thm:dDeltaDec} and Corollary \ref{cor:DDeltaDecPoly} in this setting take the following form.

\begin{theorem}[Octonionic $\dibar\Delta$-decomposition]\label{thm:dDeltaDecO}
Let $f\in\sr(\OO)$, with $\OO$ an axially symmetric open subset of $\oo$.  Assume that $\OO$ is a star-like domain with centre 0.
%If $A$ is not associative, assume that $f$ and $xf$ are $M$-admissible. Let $m=\dim(M)-1\ge3$ be odd. 
Then there exist three octonionic slice functions $f_0,f_1,f_2$ on $\OO$ in the kernel of the operator $\dibar\Delta$ such that%\marginpar{unique?} 
\[
f(x)=f_0(x)+|x|^2f_1(x)+|x|^4f_2(x)\text{\quad  $\forall x\in\OO$.}
\]
%In particular, the functions $f_0,f_1,f_2$ are biharmonic.
In particular, if $f=\sum_{n=0}^Nx^na_n\in \oo[x]$ is a polynomial with octonionic coefficients, then  
\[f_k:=-\tfrac14\big(\Pi_k(\dibar(xf))-\overline x\,\Pi_k(\dibar f)\big) \text{\quad for $k=0,1,2$}\] 
is an octonionic polynomial slice function in $x_0,\ldots, x_7$ such that $\dibar\Delta f_k=0$ and 
\[
f(x)=f_0(x)+|x|^2f_1(x)+|x|^4f_2(x)\text{\quad  $\forall x\in\oo$.}
\]
If $\deg(f)=N$, then for $k=0,1,2$ it holds $\deg(f_k)=N-2k$ or $f_k\equiv0$. The $\dibar\Delta$-Fueter mapping sends $f$ to the sequence of axially left-monogenic polynomial functions $\Fdd(f)=(g_0,g_1,g_2)$, with $g_k=\frac43\,\partial\Pi_k(\dibar f)=-4\,\partial_\B\Pi_k(\sd f)\in\ker\dibar$ for $k=0,1,2$. 
\end{theorem}

\begin{example}
For $i=1,\ldots,7$, let $Z_i=x_i-x_0e_i$ be the octonionic Fueter polynomials, and consider the axially left-monogenic polynomial
\[
g_0(x)=\sum_{i=1}^7Z_i^2=-7x_0^2+x_1^2+\cdots+x_7^2-2x_0(x_1e_1+\cdots+x_7e_7)
\]
(an example taken from \cite{Krausshar2021}).
The function $g_0$ is the slice function on $\oo$ induced by the stem function $G_0(z)=G_0(\alpha+i\beta)=-7\alpha^2+\beta^2-2\ui\alpha\beta$. A direct computation using formula \eqref{eq:pik} shows that the triple $(g_0,g_1,g_2)$, where $g_1(x)=2/3$ and $g_2(x)=0$, is in the image of the $\dibar\Delta$-Fueter mapping, namely, 
\[\Fdd\big(\tfrac5{24}x^4\big)=(g_0,g_1,g_2).\]
Since $\oo$ has the same dimension of the paravector space of $\R_7$ ($m=7$ in both cases) and the function $x^4$ is slice-preserving, the $\dibar\Delta$-decomposition of $x^4$ has the same form as the one obtained in the case of the Clifford algebra $\R_7$ (see the previous Example \ref{ex:r7}). 
\end{example}

\subsubsection*{Reduced quaternions}
Let $M=\R^3\subseteq\R_2$ be the paravector subspace of the Clifford algebra $\R_2\simeq\hh$. $M$ can be identified with the set of \emph{reduced quaternions}, i.e., the quaternions of the form $x=x_0+ix_1+jx_2$, with $x_0,x_1,x_2\in\R$. 
If $\B=(1,i,j)$, the operator $\dB$ is the Cauchy-Riemann operator of $\R_2$
\[\dibar=\frac12\left(\dd{}{x_0}+i\dd{}{x_1}+j\dd{}{x_2}\right)\]
and $\Delta_\B$ is the Euclidean Laplacian $\Delta_3$ of $\R^{3}$. 

\begin{example}
Since $m=2$, it holds $c_m=-1/2$. %Let $\B=(1,e_1,e_2)$ and $\Delta_\B=\Delta_3$. 
Consider again the power $x^3$ of the Clifford (or reduced quaternion) variable $x=x_0+x_1 e_1+x_2 e_2$ on $\R^3$.   Corollary \ref{cor:zonal_decomposition} gives the zonal decomposition
\[
x^3=h_1(x)-\overline x h_2(x)\text{\quad  $\forall x\in\R^3$},
\]
with $h_1,h_2$ given by 
\begin{align*}
h_1(x)&=-2\dibar(x^4)=4x_0(x_0^2-x_1^2-x_2^2)
,\quad\text{and}\\
h_2(x)&=-2\dibar(x^3)=3x_0^2-x_1^2-x_2^2,
\end{align*}
while the decomposition $x^3=g_1(x)-\overline x g_2(x)$ of Proposition \ref{pro:AMDec} has components
\begin{align*}
g_1(x)&=-\tfrac1{8}\Delta(x^5)=\tfrac52\left(x_0^3-x_0(x_1^2+x_2^2)+x_0^2(x_1 e_1+x_2 e_2)\right)-\tfrac12(x_1^2+x_2^2)(x_1 e_1+x_2 e_2)
,\quad\text{and}\\
g_2(x)&=-\tfrac1{8}\Delta(x^4)=\tfrac12\left(3x_0^2-x_1^2-x_2^2+2x_0(x_1 e_1+x_2 e_2)\right)
,
\end{align*}
where $\Delta=\Delta_3$. 
Note that in the exceptional case ($m=2$) both decompositions give functions in the kernel of the fractional Laplacian $(-\Delta)^{1/2}$ on $\R^3$. %For example, 
\end{example}

If one takes instead the hypercomplex subspace $M'$ of $\hh$ generated by $1,j,k$, with basis $\mathcal C=(1,-k,j)$, then $\dibar_{\mathcal C}=\frac12\left(\partial_0-k\partial_1+j\partial_2\right)$ is a multiple of the \emph{Moisil-Teodorescu operator} \cite{MT1931} in the variables $x_0,x_1,x_2$
\[
\mathcal D_{MT}=i\dd{}{x_0}+j\dd{}{x_1}+k\dd{}{x_2}=2i\dibar_{\mathcal C}.
\]

\section*{Appendix: proofs of Propositions \ref{pro:powers}, \ref{pro:gamma} and Theorem \ref{teo:difference}}\label{sec:App}

\begin{proof}[Proof of Proposition \ref{pro:powers}]
We prove (a) by induction on $n$. For $n=0,1$ the equality is true, since $\sd{x}=1$. Assume that (a) holds for the power $x^{n-1}$, with $n\ge2$. Then
\begin{align*}
2\dB(x^n)&=\textstyle\sum_iv_i\partial_ix^n=\sum_iv_i\left((\partial_i x^{n-1})x+x^{n-1}(\partial_ix)\right)\\
&=\textstyle\left(\sum_iv_i\partial_i x^{n-1}\right)x+\sum_i v_i\left(x^{n-1}(\partial_ix)\right)
=2(\dB(x^{n-1}))x+\sum_i v_i\left(x^{n-1}v_i\right),
\end{align*}
where we used Artin's Theorem to get $v_i\left((\partial_i x^{n-1})x\right)=\left(v_i\partial_i x^{n-1}\right)x$, since $\partial_i x^{n-1}\in\cc_I$ for every $x\in\cc_I$. By the inductive hypothesis,
\[
2\dB(x^n)=(1-m)\sd{(x^{n-1})}\,x+x^{n-1}+\sum_{i=1}^mv_ix^{n-1}v_i.
\]
Since $\vs{(x^{n-1})}$ and $\sd{(x^{n-1})}$ are real-valued, the last term is equal to
\begin{align*}
&\sum_{i=1}^mv_ix^{n-1}v_i=\sum_{i=1}^mv_i\left(\vs{(x^{n-1})}+\IM(x)\sd{(x^{n-1})}\right)v_i=-m\vs{(x^{n-1})}+\sd{(x^{n-1})}\sum_{i=1}^mv_i \IM(x)v_i\\
&=-m\vs{(x^{n-1})}+\sd{(x^{n-1})} (m-2)\IM(x)
\end{align*}
for every $x=x_0+x_1v_1+\cdots+x_mv_m\in M$. Therefore
\begin{align*}
2\dB(x^n)&=(1-m)\sd{(x^{n-1})}\,x+x^{n-1}-m\vs{(x^{n-1})}+\sd{(x^{n-1})} (m-2)\IM(x)\\
&=(1-m)\left(\sd{(x^{n-1})}\,x+\vs{(x^{n-1})}-\sd{(x^{n-1})}\IM(x)\right).
\end{align*}
On the other hand, $\sd{(x^n)}=\sd{(x^{n-1})}\vs{(x)}+\vs{(x^{n-1})}\sd{(x)}=\sd{(x^{n-1})}\,x_0+\vs{(x^{n-1})}=
\sd{(x^{n-1})}\,x-\sd{(x^{n-1})}\,\IM(x)+\vs{(x^{n-1})}$ and the first equality in (a) is proved. The second equality is immediate from the definition of $\sd{(x^n)}$ as $(2\IM(x))^{-1}(x^n-(x^c)^n)=(x-x^c)^{-1}(x^n-(x^c)^n)$.

We now prove (b).
Since $A$ is alternative and $v_iv_j+v_jv_i=0$ for all $1\le i,j\le m$, $i\ne j$, then
\[
v_i\partial_i(v_j\partial_jf)+v_j\partial_j(v_i\partial_if)=v_i(v_j(\partial_i\partial_jf))+v_j(v_i(\partial_j\partial_if))=-(v_i,v_j,\partial_i\partial_jf)-(v_j,v_i,\partial_i\partial_jf)=0.
\]
Moreover, $\partial_0(v_i\partial_if)-v_i(\partial_i\partial_0f)=0$ for every $i\ge1$. 
It follows that 
\[
4\partial_\B\dB=4\dB\partial_\B=\partial_0(\partial_0 f)-\sum_{i=1}^m v_i\partial_i(v_i\partial_if)=\sum_{i=0}^m \partial_i(\partial_if)=\Delta_\B f,
\]
\end{proof}

\begin{proof}[Proof of Proposition \ref{pro:gamma}]
To prove point (a) of the Proposition, we compute
\[\GB(x)=-\tfrac12\sum_{i,j=1}^mv_i(v_j(x_iv_j-x_jv_i))
=\tfrac12\sum_{i\ne j}x_iv_i+\tfrac12\sum_{i\ne j}x_jv_j=(m-1)\IM(x),
\]
where we used Artin's Theorem. We now prove point (b). A direct computation shows that the equalities
\begin{equation}\label{eq:lij}
L_{ij}(xf)=(L_{ij}x)f+x\,L_{ij}f
\end{equation}
hold true on $M$ for every $i,j$ and every $f$. This implies that on $M$
\begin{equation}\label{eq:gbxf}
\GB(xf)=-\tfrac12\sum_{i,j}v_i(v_j L_{ij}(xf))=-\tfrac12\sum_{i,j}v_i(v_j((L_{ij}x)f(x)))-\tfrac12\sum_{i,j}v_i(v_j(xL_{ij}f)).
\end{equation}
Since $L_{ij}x=x_iv_j-x_jv_i$, using the first Moufang identity \eqref{moufang1} and (a) we get that the first sum in the last term of \eqref{eq:gbxf} is equal to 
\[
%-\sum_{i<j}v_i(v_j((x_iv_j-x_jv_i)f(x)))=
-\tfrac12\sum_{i,j}v_i(v_j((x_iv_j-x_jv_i)))f(x)=\GB(x)f(x)=(m-1)\IM(x)f(x).
\]
Now we consider the second sum in last term of \eqref{eq:gbxf}. We have
\begin{align*}
&\tfrac12\sum_{i,j}v_i(v_j(xL_{ij}f))+x^c\,\GB f\\
&=\tfrac12\sum_{i,j}v_i(v_j((x_0+\sum_{k=1}^m x_kv_k)(x_i\dds jf-x_j\dds if)))-
(x_0-\sum_{k=1}^mx_kv_k)\,\tfrac12\sum_{i,j} v_i(v_j(x_i\dds jf-x_j\dds if))\\
&=\tfrac12\sum_{i\ne j}\sum_{k=1}^m v_i(v_j((x_kv_k)(x_i\dds jf-x_j\dds if)))+
\tfrac12\sum_{k=1}^mx_kv_k\,\sum_{i,j} v_i(v_j(x_i\dds jf-x_j\dds if)).
\end{align*}
In the last expression we can assume that the index $k$ is different from $i$ and $j$, since from Artin's Theorem and the first Moufang identity \eqref{moufang1}, it follows that for every $i\ne j$ and every $a\in A$, it holds
\begin{equation}\label{eq:iji}
v_i(v_j(v_ia))=(v_iv_jv_i)a=v_ja=-v_i(v_i(v_ja))
\end{equation}
and
\begin{equation}\label{eq:ijj}
v_i(v_j(v_ja))=-v_ia=-(v_jv_iv_j)a=-v_j(v_i(v_ja)).
\end{equation}
Therefore
\begin{align}\label{eq:zero}
&\tfrac12\sum_{i,j}v_i(v_j(xL_{ij}f))+x^c\,\GB f\notag\\\notag
% &=\tfrac12\sum_{i\ne j}\sum_{k\ne i,j}\left(x_k v_i(v_j(v_k(x_i\dds jf-x_j\dds if)))+
% x_kv_k (v_i(v_j(x_i\dds jf-x_j\dds if)))\right)\\
&=\tfrac12\sum x_k v_i(v_j(v_k(x_i\dds jf)))-\tfrac12\sum x_k v_i(v_j(v_k(x_j\dds if)))\\\notag
&\quad +\tfrac12\sum x_kv_k (v_i(v_j(x_i\dds jf)))-\tfrac12\sum x_kv_k (v_i(v_j(x_j\dds if)))\\
&=\tfrac12\sum\nolimits' x_ix_k\left(v_i(v_j(v_k\dds j f))-v_j(v_i(v_k\dds jf))+
v_k(v_i(v_j\dds jf))-v_k(v_j(v_i\dds jf))
\right)
\end{align}
where the symbol $\sum\nolimits'$ indicates that the sum is made over three mutually different indices $i,j,k$ in the set $\{1,\ldots,m\}$. 
If $A$ is associative, the terms in \eqref{eq:zero} are all equal and the sum reduces to
\[2\sum\nolimits' x_ix_k v_k v_i v_j\dds jf
\]
which is zero by antisymmetry w.r.t.\ the indices $i$ and $k$. 

We now consider the general case.
The first and the last terms in \eqref{eq:zero} cancel out in the sum, since they are symmetric w.r.t.\ the indices $i$ and $k$. 
In order to show that also the other terms vanish, we first prove that for every $x\in M$, it holds
\begin{equation}\label{eq:ik}
v_i(v_k x)=-v_k(v_i x)\text{\quad if $1\le i,k\le m, i\ne k$}.
\end{equation}
Let $x=x'+x''$, with $x'\in\Span(v_i,v_k)$, $x''\in\Span(v_i,v_k)^\bot\cap M$. Then
\[
v_i(v_kx')=-x'(v_iv_k)=x'(v_kv_i)=-v_k(v_ix')
\]
thanks to Artin's Theorem, as in  \eqref{eq:iji} and \eqref{eq:ijj} with $a=1$. Since $\Span(v_i,v_k)^\bot\cap M$ is generated as a real vector subspace by $\{1,v_1,\ldots,v_m\}\setminus\{v_i,v_k\}$, it suffices to prove that
\[
v_i(v_kv_j)=-v_k(v_iv_j)
\]
for every $j\in\{1,\ldots,m\}$, $j\ne i,k$, $i\ne k$. Using the alternating property of the associator $(a,b,c)=(ab)c-a(bc)$ in $A$, we get
\[
v_i(v_kv_j)+v_k(v_iv_j)=(v_iv_k)v_j-(v_i,v_k,v_j)+(v_kv_i)v_j-(v_k,v_i,v_j)=0.
\]
Since $\dds jf(x)\in M$ for every $j$, from \eqref{eq:ik} it follows that also the second sum in \eqref{eq:zero} vanishes. 
It remains to consider the third sum in \eqref{eq:zero}, i.e.,
\begin{equation}\label{eq:third}
\tfrac12\sum_{i\ne k} x_ix_k v_k(v_i(\sum_{\stackrel{j=1}{j\ne i,k}}^m v_j\dds jf))=\sum_{i\ne k} x_ix_k v_k(v_i(\dB f-\tfrac12\dds 0f-\tfrac12 v_i\dds if-\tfrac12 v_k\dds kf)).
\end{equation}
In view of \eqref{eq:ik}, since by assumption $\dB f(x)$ and $\dds 0f(x)$ belong to $M$ for every $x\in\OO$, the sum 
\[\sum_{i\ne k} x_ix_k v_k(v_i(\dB f-\tfrac12\dds 0f))
\]
vanishes. Moreover, using again \eqref{eq:iji}  and \eqref{eq:ijj}, we get
\[
\sum_{i\ne k} x_ix_k v_k(v_i(v_i\dds if+ v_k\dds kf))=\sum_{i\ne k} x_ix_k (-v_k\dds if+ v_i\dds kf)=0
\]
and this concludes the proof of (b) also in the non-associative case.
\end{proof}

We now come to the proof of Theorem \ref{teo:difference}:

\begin{proof}[Proof of Theorem \ref{teo:difference}]
From definitions, on $\OO\setminus\R$ it holds
\[
2(\dB f-\difbar f)=\sum_{i=1}^m v_i\dds if-\frac{\IM(x)}{n(\IM(x))} \sum_{i=1}^m x_i\partial_if.
\]
We compute
\begin{align*}
&\IM(x)\left(\sum_{i=1}^m x_i\partial_if+\Gamma_\B f\right)=\sum_{k=1}^mx_kv_k\sum_{i=1}^m x_i\partial_if-\tfrac12\sum_{k=1}^mx_kv_k\sum_{i,j}v_i(v_j(L_{ij}f))\\
&=\sum_{i,k}x_kx_iv_k\dds if-\tfrac12 \sum_{i,j,k}x_kx_i v_k(v_i(v_j\dds jf))+\tfrac12 \sum_{i,j,k}x_kx_j v_k(v_i(v_j\dds if))\\
&=\sum_{i,k}x_kx_iv_k\dds if+\tfrac12 \sum_{k,j}x_k^2 v_j\dds jf+\tfrac12 \sum_{\stackrel{i,j,k}{i\ne j}}x_kx_j v_k(v_i(v_j\dds if))-\tfrac12 \sum_{i,k}x_kx_i v_k\dds if\\
&=\sum_{i,k}x_kx_iv_k\dds if+\tfrac12 \sum_{k}x_k^2\sum_j v_j\dds jf+\tfrac12 \sum_{\stackrel{i,j,k}{i\ne j,j\ne k}}x_kx_j v_k(v_i(v_j\dds if))+\tfrac12 \sum_{\stackrel{i,k}{i\ne k}}x_k^2 v_k(v_i(v_k\dds if))\\
&\quad-\tfrac12 \sum_{i,k}x_kx_i v_k\dds if\\
&=\sum_{i,k}x_kx_iv_k\dds if+\tfrac12 \sum_{k}x_k^2\sum_j v_j\dds jf+\tfrac12 \sum\nolimits' x_kx_j v_k(v_i(v_j\dds if))+
\tfrac12\sum_{i\ne j} x_ix_j v_i(v_i(v_j\dds if))\\
&\quad+\tfrac12 \sum_{\stackrel{i,k}{i\ne k}}x_k^2 v_i\dds if-\tfrac12 \sum_{i,k}x_kx_i v_k\dds if
\end{align*}
\begin{align*}
&=\tfrac12\sum_{i,k}x_kx_iv_k\dds if+\sum_{k}x_k^2\sum_j v_j\dds jf-\tfrac12 \sum\nolimits'x_kx_j v_k(v_j(v_i\dds if))-
\tfrac12\sum_{i\ne j} x_ix_j v_j\dds if\\
&\quad-\tfrac12 \sum_{i}x_i^2 v_i\dds if\\
&=n(\IM(x))\sum_j v_j\dds jf-\tfrac12 \sum\nolimits' x_kx_j v_k(v_j(v_i\dds if)),
\end{align*}
where the symbol $\sum\nolimits'$ denotes a sum over all distinct indices $i,j,k$ in the set $\{1,\ldots,n\}$.

If $A$ is associative, then the sum $\sum\nolimits' x_kx_j v_k(v_j(v_i\dds if))$ vanishes, since it is antisymmetric w.r.t.\ the indices $j$ and $k$. In the general case, the vanishing of this sum follows from the same arguments used  in Proposition \ref{pro:gamma} when it was proved that the sum \eqref{eq:third} is zero. 
Therefore we have
\[
\frac{\IM(x)}{n(\IM(x))}\left(\sum_{i=1}^m x_i\partial_if+\Gamma_\B f\right)=\sum_{j=1}^m v_j\dds jf\text{\quad on $\OO\setminus\rr$}
\]
and the thesis
\[
2(\dB f-\difbar f)=\sum_{i=1}^m v_i\dds if-\frac{\IM(x)}{n(\IM(x))} \sum_{i=1}^m x_i\partial_if=\frac{\IM(x)}{n(\IM(x))} \Gamma_\B f=(\IM(x))^{-1}\Gamma_\B f.
\]
\end{proof}

\section*{Aknowledgments}
The author is a member the INdAM Research group GNSAGA and was supported by the grants ``Progetto di Ricerca INdAM, Teoria delle funzioni ipercomplesse e applicazioni'', and PRIN ``Real and Complex Manifolds: Topology, Geometry and holomorphic dynamics''.

% \bibliographystyle{abbrv}
% \bibliography{Ref_AP}

\end{document}